\documentclass[11pt,a4paper,leqno]{amsart}

\title
[Anisotropic global microlocal analysis for tempered distributions]
{Anisotropic global microlocal analysis for tempered distributions}

\author[L. Rodino]{Luigi Rodino}

\address{Department of Mathematics, Universit\`a di Torino, Via Carlo Alberto 10,
10123 Torino, Italy}

\email{luigi.rodino[AT]unito.it}

\author[P. Wahlberg]{Patrik Wahlberg}

\address{Dipartimento di Scienze Matematiche, Politecnico di Torino, Corso Duca degli Abruzzi 24,
10129 Torino, Italy}

\email{patrik.wahlberg[AT]polito.it}

\usepackage{latexsym}
\usepackage[a4paper]{geometry}
\usepackage{amsmath}
\usepackage{amssymb}
\usepackage{amsthm}
\usepackage{bbm}
\usepackage{amsfonts}
\usepackage{mathrsfs}
\usepackage{calc}
\usepackage{cite}
\usepackage{color}
\usepackage{graphicx}
\usepackage{enumerate}

\numberwithin{equation}{section}          

\newtheorem{thm}{Theorem}
\numberwithin{thm}{section}

\newcommand{\rubrik}{}
\newtheorem{prop}[thm]{Proposition}
\newtheorem{cor}[thm]{Corollary}
\newtheorem{lem}[thm]{Lemma}

\theoremstyle{definition}

\newtheorem{defn}[thm]{Definition}
\newtheorem{example}[thm]{Example}

\theoremstyle{remark}

\newcommand{\scal}[2]{\langle #1,#2\rangle}
\newcommand{\pd}[1] {\partial ^#1}
\newcommand{\pdd}[2] {\partial_{#1} ^{#2}}

\newcommand{\ro}{\mathbf R}
\newcommand{\no}{\mathbf N}
\newcommand{\qo}{\mathbf Q}
\newcommand{\rr}[1]{\mathbf R^{#1}}
\newcommand{\sr}[1]{\mathbf S^{#1}}
\newcommand{\sro}[1]{\mathbf S}
\newcommand{\nn}[1]{\mathbf N^{#1}}

\newcommand{\co}{\mathbf C}

\newcommand{\dd}{\mathrm {d}}

\newcommand{\ep}{\varepsilon}
\newcommand{\fy}{\varphi}

\newcommand{\cdo}{\, \cdot \, }

\newcommand{\supp}{\operatorname{supp}}

\newcommand{\wpr}{{\text{\footnotesize $\#$}}}

\newcommand{\eabs}[1]{\langle #1\rangle}

\newcommand{\Sp}{\operatorname{Sp}}

\newcommand{\GL}{\operatorname{GL}}

\newcommand{\charac}{\operatorname{char}}

\newcommand{\csupp}{\operatorname{conesupp}}
\newcommand{\rB}{\operatorname{B}}

\newcommand{\WF}{\mathrm{WF}}
\newcommand{\WFg}{\mathrm{WF_g}}
\newcommand{\WFgs}{\mathrm{WF_{g}^{\it s}}}

\newcommand{\cS}{\mathscr{S}}

\newcommand{\cF}{\mathscr{F}}

\newcommand{\J}{\mathcal{J}}
\newcommand{\wt}{\widetilde}
\newcommand{\wh}{\widehat}

\def\la{\langle}
\def\ra{\rangle}
\newcommand{\leqs}{\leqslant}
\newcommand{\geqs}{\geqslant}

\begin{document}

\begin{abstract}
We study an anisotropic version of the Shubin calculus of pseudodifferential operators on $\rr d$. 
Anisotropic symbols and Gabor wave front sets are defined in terms of decay or growth along 
curves in phase space of power type parametrized by one positive parameter that distinguishes space and frequency variables. 
We show that this gives subcalculi of Shubin's isotropic calculus, and we show a microlocal as well as a microelliptic 
inclusion in the framework. 
Finally we prove an inclusion for the anisotropic Gabor wave front set of chirp type oscillatory functions with a real polynomial phase function.  
\end{abstract}

\keywords{Tempered distributions, global wave front sets, pseudodifferential operators, Shubin calculus, microlocality, microellipticity, phase space, anisotropy}
\subjclass[2010]{46F05, 46F12, 35A27, 35S05, 35A18, 81S30, 58J47}

\maketitle

\section{Introduction}\label{sec:intro}

In this paper we study an anisotropic version of Shubin's calculus of pseudodifferential operators on $\rr d$ \cite{Shubin1}
and a naturally appearing anisotropic Gabor wave front set. 

Shubin symbols for pseudodifferential operators satisfy estimates involving $1+|x| + |\xi|$, and they are thus isotropic on the phase space $(x,\xi) \in T^* \rr d$. In particular they behave in a way that does not distinguish between $x \in \rr d$ and $\xi \in \rr d$. 

Otherwise expressed, the symbols satisfy growth or decay restrictions on straight lines in phase space of the form 
$\ro_+ \ni \lambda \mapsto (\lambda x, \lambda \xi)$ for $(x,\xi) \in T^* \rr d \setminus 0$. 
For Shubin operators there are results concerning global microlocal analysis involving the Gabor wave front set, introduced by
H\"ormander in \cite{Hormander1} and elaborated in several recent works \cite{Boiti1,Cappiello1,Cappiello3,Carypis1,Cordero1,PRW1,Rodino2,Rodino4,Schulz1,Wahlberg2}. 
The Gabor wave front set detects the lack of superpolynomial decay along straight lines in phase space of the short-time Fourier transform of a tempered distribution. 
It is global in the sense that it measures smoothness and decay at infinity of the distribution comprehensively. It is empty exactly when a tempered distribution is a Schwartz function. 

In this paper we replace the weight $1+|x| + |\xi|$ by $1+|x| + |\xi|^{\frac{1}{s}}$ where $s > 0$. 
We introduce Shubin type symbols with anisotropic behaviour, with decay or growth along power type curves 
in phase space of the form
\begin{equation}\label{eq:curvephasespace}
\ro_+ \ni \lambda \mapsto (\lambda x, \lambda^s \xi)
\end{equation}
for $(x,\xi) \in T^* \rr d \setminus 0$. 

The idea of anisotropy in pseudodifferential calculus has been around for a long time, cf. \cite{Lascar1,Parenti1},
with recent contributions exemplified by \cite{Garello1}. 
These works treat mainly anisotropic behavior in the frequency variable $\xi \in \rr d$ with $d$ parameters, for fixed $x \in \rr d$. 
Our idea is to study global anisotropy comprehensively in the phase space $T^* \rr d$. 
For simplicity we use only one parameter for the relation between the space and the frequency variables. 
Even the idea of anisotropic pseudodifferential calculus on $T^* \rr d$ is not new, cf. \cite{Cappiello0,Cappiello4,Martin1,Parenti1}, 
but as far as we know a systematic microlocal analysis has not yet been fully developed. The aim of our paper is to contribute to such a calculus and adapted microlocal analysis.  

For $s > 0$ and $m \in \ro$ we study symbols that are smooth and satisfy estimates of the form 
\begin{equation*}
|\pdd x \alpha \pdd \xi \beta a(x,\xi)|
\lesssim ( 1 + |x| + |\xi|^{\frac1s} )^{m - |\alpha| - s |\beta|}, \quad (x,\xi) \in T^* \rr d, \quad \alpha, \beta \in \nn d. 
\end{equation*}
This is a generalization of the isotropic Shubin symbols that satisfy the estimates with $s=1$. 
When $s \neq 1$ an anisotropic symbol is still embedded in an isotropic Shubin symbol space of possibly higher order.
These symbol classes were introduced in \cite[Definition~3.1]{Parenti1}, 
and the corresponding basic calculus is briefly stated there without proofs. 
In this paper we provide detailed proofs of the calculus from scratch, and extend the analysis to an adapted anisotropic Gabor microlocal analysis. 

For fixed $s > 0$ we show that the anisotropic symbols give rise to a subcalculus  of the isotropic Shubin calculus. 
More precisely the anisotropic symbol classes are independent of the quantization parameter that admits transfer between 
Weyl and Kohn--Nirenberg quantization. 
They are also stable with respect to operator composition as well as formal adjoint. 

Then we introduce the corresponding notion of anisotropic Gabor wave front set $\WFgs (u)$ of a tempered distribution $u$. 
This means the complement of curves of the form \eqref{eq:curvephasespace} 
in a neighborhood of which the short-time Fourier transform decays superpolynomially. 
The neighborhoods are $s$-conic, that is if a point $(x,\xi) \in T^* \rr d \setminus 0$ belongs to the neighborhood then it contains the whole curve  \eqref{eq:curvephasespace}, 
and so is the anisotropic Gabor wave front set. 

The first main result that we present is the microlocal inclusion 
\begin{equation*}
\WFgs ( a^w(x,D) u) \subseteq \WFgs ( u),   
\end{equation*}
where $u$ is a tempered distribution, $a$ is an isotropic Shubin symbol, and $a^w(x,D)$ denotes the 
Weyl quantization. 

The second main result is the microelliptic inclusion 
\begin{equation*}
\WFgs (u) \subseteq \WFgs ( a^w(x,D) u ) \bigcup \charac_{s,m_1}(a)
\end{equation*}
where again $u$ is a tempered distribution, 
$a$ is an anisotropic Shubin symbol with parameter $s > 0$ and order $m$, $m_1 \leqs m$
and $\charac_{s,m_1}(a)$ is a notion of microlocal characteristic set adapted to the anisotropic Shubin calculus
(see Definition \ref{def:noncharacteristic}). 

Taken together these results imply 
\begin{equation*}
\WFgs ( a^w(x,D) u ) = \WFgs (u) 
\end{equation*}
if $\charac_{s,m_1}(a) = \emptyset$ for some $m_1 \leqs m$.

The paper is organized as follows. 
Section \ref{sec:prelim} sets the stage in terms of notations and some definitions, and a background 
on pseudodifferential operators in the Weyl quantization with isotropic Shubin symbols. 
In Section \ref{sec:anisotropicshubin} we introduce the anisotropic Shubin symbols for a fixed parameter $s > 0$. 
We show adapted asymptotic expansions, and invariance under a commonly used family of quantizations 
parametrized by a real parameter. This family includes the Weyl as well as the Kohn--Nirenberg quantization. 
We also show the continuity of the Weyl product acting on the anisotropic symbol Fr\'echet spaces, 
and we discuss $s$-conic cutoff functions. 

Section \ref{sec:anisotropicgaborWF} is devoted to the anisotropic Gabor wave front set. 
We state the definition, discuss a few properties and show that it does not depend of the chosen nonzero Schwartz 
window function in the short-time Fourier transform. 
The full metaplectic invariance of the isotropic ($s=1$) Gabor wave front set does not hold when $s \neq 1$ but
we show a few partial such invariances. 

In Section \ref{sec:microlocalgabor} we show that pseudodifferential operators with isotropic Shubin symbols are microlocal with respect to 
all anisotropic Gabor wave front sets. In particular microlocality holds for anisotropic Shubin symbols. 
Another consequence is the invariance of anisotropic Gabor wave front sets with respect to translation and modulation. 

Section \ref{sec:microellipticgabor} treats a microelliptic inclusion for the anisotropic Gabor wave front set and anisotropic Shubin symbols
with $s > 0$ fixed. 
Finally in Section \ref{sec:chirp} we show inclusions and equalities for the anisotropic Gabor wave front set of oscillatory functions with phase functions that are real polynomials on $\rr d$ of order $m \geqs 2$. 
The anisotropy parameter is $s = m-1$. 

\section{Preliminaries}\label{sec:prelim}

The unit sphere in $\rr d$ is denoted by $\sr {d-1} \subseteq \rr d$. 
A ball of radius $r > 0$ in $\rr d$ is denoted by $\rB_r$, 
and $e_j \in \rr d$ is the vector of zeros except for position $j$, $1 \leqs j \leqs d$, where it is one. 
The transpose of a matrix $A \in \rr {d \times d}$ is denoted by $A^T$. 
We write $f (x) \lesssim g (x)$ provided there exists $C>0$ such that $f (x) \leqs C \, g(x)$ for all $x$ in the domain of $f$ and of $g$. 
If $f (x) \lesssim g (x) \lesssim f(x)$ then we write $f \asymp g$. 
We use the bracket $\eabs{x} = (1 + |x|^2)^{\frac12}$ for $x \in \rr d$. 
Peetre's inequality with optimal constant \cite[Lemma~2.1]{Rodino3} is
\begin{equation*}
\eabs{x+y}^s \leqs \left( \frac{2}{\sqrt{3}} \right)^{|s|} \eabs{x}^s\eabs{y}^{|s|}\qquad x,y \in \rr d, \quad s \in \ro. 
\end{equation*}
The normalization of the Fourier transform is
\begin{equation*}
 \cF f (\xi )= \widehat f(\xi ) = (2\pi )^{-\frac d2} \int _{\rr
{d}} f(x)e^{-i\scal  x\xi }\, \dd x, \qquad \xi \in \rr d, 
\end{equation*}
for $f\in \cS(\rr d)$ (the Schwartz space), where $\scal \cdo \cdo$ denotes the scalar product on $\rr d$. 
The conjugate linear action of a tempered distribution $u \in \cS'(\rr d)$ on a test function $\phi \in \cS(\rr d)$ is written $(u,\phi)$, consistent with the $L^2$ inner product $(\cdo ,\cdo ) = (\cdo ,\cdo )_{L^2}$ which is conjugate linear in the second argument. 

Denote translation by $T_x f(y) = f( y-x )$ and modulation by $M_\xi f(y) = e^{i \scal y \xi} f(y)$ 
for $x,y,\xi \in \rr d$ where $f$ is a function or distribution defined on $\rr d$. 
The composed operator is denoted by $\Pi(x,\xi) = M_\xi T_x$. 
Let $\fy \in \cS(\rr d) \setminus \{0\}$. 
The short-time Fourier transform (STFT) of a tempered distribution $u \in \cS'(\rr d)$ is defined by 
\begin{equation*}
V_\fy u (x,\xi) = (2\pi )^{-\frac d2} (u, M_\xi T_x \fy) = \cF (u T_x \overline \fy)(\xi), \quad x,\xi \in \rr d. 
\end{equation*}
The function $V_\fy u$ is smooth and polynomially bounded \cite[Theorem~11.2.3]{Grochenig1}, that is 
there exists $k \geqs 0$ such that 
\begin{equation}\label{eq:STFTtempered}
|V_\fy u (x,\xi)| \lesssim \eabs{(x,\xi)}^{k}, \quad (x,\xi) \in T^* \rr d.  
\end{equation}
We have $u \in \cS(\rr d)$ if and only if
\begin{equation}\label{eq:STFTschwartz}
|V_\fy u (x,\xi)| \lesssim \eabs{(x,\xi)}^{-N}, \quad (x,\xi) \in T^* \rr d, \quad \forall N \geqs 0.  
\end{equation}

The inverse transform is given by
\begin{equation}\label{eq:STFTinverse}
u = (2\pi )^{-\frac d2} \iint_{\rr {2d}} V_\fy u (x,\xi) M_\xi T_x \fy \, \dd x \, \dd \xi
\end{equation}
provided $\| \fy \|_{L^2} = 1$, with action under the integral understood, that is 
\begin{equation}\label{eq:moyal}
(u, f) = (V_\fy u, V_\fy f)_{L^2(\rr {2d})}
\end{equation}
for $u \in \cS'(\rr d)$ and $f \in \cS(\rr d)$, cf. \cite[Theorem~11.2.5]{Grochenig1}. 

We will use
\begin{equation*}
|x + y|^{\frac1s} \leqs \kappa(s^{-1}) \left(  |x|^{\frac1s} + |y|^{\frac1s} \right), \quad x,y \in \rr d, \quad s > 0,
\end{equation*}
where
\begin{equation*}
\kappa (t ) =
\left\{
\begin{array}{ll}
1 & \mbox{if} \quad 0 < t \leqs 1 \\
2^{t-1} & \mbox{if} \quad t > 1
\end{array}
\right. .
\end{equation*}

Let $s > 0$. We use the weight function on $(x,\xi) \in T^* \rr d$
\begin{equation}\label{eq:weightanisotrop}
\mu_s(x,\xi) = 1 + |x| + |\xi|^{\frac1s}. 
\end{equation}

The following inequality of Peetre type holds.

\begin{lem}\label{lem:speetre}
If $t \in \ro$ then 
\begin{equation*}
\mu_s( x+y, \xi + \eta)^t 
\leqs C_{s,t} \mu_s( x, \xi)^{|t|} \mu_s( y, \eta)^t, \quad x, y, \xi, \eta \in \rr d. 
\end{equation*}
\end{lem}

\begin{proof}
We may assume $t = 1$. 
We have 
\begin{align*}
\mu_s( x+y, \xi + \eta)
& = 1 + |x + y| + |\xi + \eta|^{\frac1s} \\
& \leqs 1 + |x| + |y| + \kappa(s^{-1}) |\xi|^{\frac1s} + \kappa(s^{-1}) |\eta|^{\frac1s} \\
& \leqs \left(1 + |x| + \kappa(s^{-1}) |\xi|^{\frac1s} \right) \left( 1 + |y| + \kappa(s^{-1}) |\eta|^{\frac1s} \right) \\
& \leqs \kappa(s^{-1})^2 \mu_s(x,\xi) \mu_s(y,\eta). 
\end{align*}
\end{proof}

For $s > 0$ we will use subsets of $T^* \rr d \setminus 0$ that are $s$-conic, 
that is subsets closed under the operation $T^* \rr d \setminus 0 \ni (x,\xi) \mapsto ( \lambda x, \lambda^s \xi)$
for all $\lambda > 0$.

\subsection{Pseudodifferential operators}

We need some elements from the calculus of pseudodifferential operators \cite{Folland1,Hormander0,Nicola1,Shubin1}. 
Let $a \in C^\infty (\rr {2d})$, $m \in \ro$ and $0 \leqs \rho \leqs 1$. Then $a$ is a \emph{Shubin symbol} of order $m$ and parameter $\rho$, denoted $a\in G_\rho^m$, if for all $\alpha,\beta \in \nn d$ there exists a constant $C_{\alpha,\beta}>0$ such that
\begin{equation}\label{eq:shubinineq}
|\partial_x^\alpha \partial_\xi^\beta a(x,\xi)| \leqs C_{\alpha,\beta} \langle (x,\xi)\rangle^{m - \rho|\alpha + \beta|}, \quad x,\xi \in \rr d.
\end{equation}
The Shubin symbols $G_\rho^m$ form a Fr\'echet space where the seminorms are given by the smallest possible constants in \eqref{eq:shubinineq}.
We write $G_1^m = G^m$. 

For $a \in G_\rho^m$ and $t \in \ro$ a pseudodifferential operator in the $t$-quantization is defined by
\begin{equation}\label{eq:tquantization}
a_t(x,D) f(x)
= (2\pi)^{-d}  \int_{\rr {2d}} e^{i \langle x-y, \xi \rangle} a ( (1-t) x + t y,\xi ) \, f(y) \, \dd y \, \dd \xi, \quad f \in \cS(\rr d),
\end{equation}
when $m<-d$. The definition extends to $m \in \ro$ if the integral is viewed as an oscillatory integral.
If $t=0$ we get the Kohn--Nirenberg quantization $a_0(x,D)$ and if $t = \frac12$ we get the Weyl quantization $a_{1/2}(x,D) = a^w(x,D)$. 
The relation between symbols in different quantizations is \cite{Hormander0}
\begin{equation*}
e^{i t \la D_x , D_\xi \ra } a_t (x,\xi) 
= e^{i s \la D_x , D_\xi \ra } a_s (x,\xi), \quad t,s \in \ro
\end{equation*}
where $e^{i t \la D_x , D_\xi \ra }$ is the Fourier multiplier operator with symbol $e^{i t \la x , \xi \ra }$.
Using \cite[Theorem~7.6.1]{Hormander0} we may write for $t \in \ro \setminus 0$ and $a \in \cS(\rr {2d})$
\begin{equation}\label{eq:quantizationchange}
e^{i t \la D_x , D_\xi \ra } a (x,\xi) 
= (2 \pi |t| )^{-d} \iint_{\rr {2d}} a(y,\eta) e^{ - \frac{i}{t} \la x-y , \xi-\eta \ra } \dd y \, \dd \eta. 
\end{equation}

If $0 < \rho \leqs 1$ then the Shubin symbols are invariant with respect to the parameter $t$ in the sense of $a_t \in G_\rho^m$ 
if and only if $a_s = e^{i (t-s) \la D_x , D_\xi \ra } a_t \in G_\rho^m$ for any $t,s \in \ro$ \cite[Theorem~23.2]{Shubin1}. 
If $t \in \ro$ then for the formal adjoint we have $a_t(x,D)^* = \overline{a}_{1-t}(x,D)$. 
Thus if $a_t \in G_\rho^m$ then $a_t(x,D)^* = b_t(x,D)$ where $b_t \in G_\rho^m$ \cite[Theorem~23.5]{Shubin1}. 

We will use exclusively the Weyl quantization which has several particular features. 
One important such feature is the simplicity of the formal adjoint: $a^w(x,D)^* = \overline{a}^w(x,D)$.  
As for the Shubin symbols, we will see that also the anisotropic symbol classes that we will use in this paper 
give pseudodifferential calculi that are invariant with respect to the quantization parameter $t \in \ro$ 
(see Proposition \ref{prop:calculusanisotropic}). 

If $0 < \rho \leqs 1$ and $a \in G_\rho^m$
then the operator $a^w(x,D)$ acts continuously on $\cS(\rr d)$ and extends uniquely by duality to a continuous operator on $\cS'(\rr d)$.
By Schwartz's kernel theorem the Weyl quantization may be extended to a weak formulation which yields continuous linear operators $a^w(x,D):\cS(\rr{d}) \to \cS'(\rr{d})$, even if $a$ is only an element of $\cS'(\rr{2d})$.

If $a \in \cS'(\rr {2d})$ then 
\begin{equation}\label{eq:wignerweyl}
(a^w(x,D) f, g) = (2 \pi)^{-d} (a, W(g,f) ), \quad f, g \in \cS(\rr d), 
\end{equation}
where the cross-Wigner distribution \cite{Folland1,Grochenig1} is defined as 
\begin{equation*}
W(g,f) (x,\xi) = \int_{\rr d} g (x+y/2) \overline{f(x-y/2)} e^{- i \la y, \xi \ra} \dd y, \quad (x,\xi) \in \rr {2d}. 
\end{equation*}
We have $W(g,f) \in \cS (\rr {2d})$ when $f,g \in \cS(\rr d)$. 

The real phase space $T^* \rr d \simeq \rr d \oplus \rr d$ is a real symplectic vector space equipped with the 
canonical symplectic form
\begin{equation*}
\sigma((x,\xi), (x',\xi')) = \langle x' , \xi \rangle - \langle x, \xi' \rangle, \quad (x,\xi), (x',\xi') \in T^* \rr d. 
\end{equation*}
This form can be expressed with the inner product as $\sigma(X,Y) = \la \J X, Y \ra$ for $X,Y \in T^* \rr d$
where 
\begin{equation*}
\J =
\left(
\begin{array}{cc}
0 & I_d \\
-I_d & 0
\end{array}
\right) \in \rr {2d \times 2d}. 
\end{equation*}
The real symplectic group $\Sp(d,\ro)$ is the set of matrices in $\GL(2d,\ro)$ that leaves $\sigma$ invariant. 
Hence $\J \in \Sp(d,\ro)$. 

To each symplectic matrix $\chi \in \Sp(d,\ro)$ is associated an operator $\mu(\chi)$ that is unitary on $L^2(\rr d)$, and determined up to a complex factor of modulus one, such that
\begin{equation*}
\mu(\chi)^{-1} a^w(x,D) \, \mu(\chi) = (a \circ \chi)^w(x,D), \quad a \in \cS'(\rr {2d})
\end{equation*}
(cf. \cite{Folland1,Hormander0}).
The operator $\mu(\chi)$ is a homeomorphism on $\mathscr S$ and on $\mathscr S'$.

The mapping $\Sp(d,\ro) \ni \chi \rightarrow \mu(\chi)$ is called the metaplectic representation \cite{Folland1}.
It is in fact a representation of the so called $2$-fold covering group of $\Sp(d,\ro)$, which is called the metaplectic group.  
The metaplectic representation satisfies the homomorphism relation modulo a change of sign:
\begin{equation*}
\mu( \chi \chi') = \pm \mu(\chi ) \mu(\chi' ), \quad \chi, \chi' \in \Sp(d,\ro).
\end{equation*}
We do not enter into the geometric subtleties of this construction since they are not needed in this paper. 

Let $0 < \rho \leqs 1$. 
The Weyl product $a \wpr b$ of two symbols $a \in G_\rho^m$ and $b \in G_\rho^n$
is defined as the product of symbols corresponding to operator composition: 
$( a \wpr b)^w(x,D) = a^w(x,D) b^w (x,D)$. 
According to \cite[Theorem~23.6]{Shubin1} $a \wpr b \in G_\rho^{m+n}$ if $a \in G_\rho^m$ and $b \in G_\rho^n$, 
and the bilinear map $(a,b) \mapsto a \wpr b$ is continuous $G_\rho^m \times G_\rho^n \to G_\rho^{m+n}$. 
When $a,b \in \cS(\rr {2d})$ we have the formula \cite[Eq.~(18.5.6)]{Hormander0}
\begin{equation}\label{eq:weylproduct1}
a \wpr b (x,\xi) = e^{\frac{i}{2} \sigma (D_x, D_\xi; D_y, D_\eta)} a(x,\xi) b(y,\eta) \big|_{(y,\eta) = (x,\xi)}. 
\end{equation}
Using \cite[Vol.~3 p.~152]{Hormander0} we may write for $a,b \in \cS(\rr {2d})$
\begin{equation}\label{eq:weylproduct2}
a \wpr b (z) = 
\pi^{-2d} \iint_{\rr {4d}} a(w) b (u) e^{2 i \sigma (z-u, z-w)} \dd w \, \dd u, \quad z \in T^* \rr d.  
\end{equation}
%

\section{Anisotropic Shubin calculus}\label{sec:anisotropicshubin}

Let $s > 0$ be fixed. 
We need a simplified version of a tool taken from \cite{Lascar1,Parenti1} and their references. 
Given $(x,\xi) \in \rr {2d} \setminus 0$ there is a unique $\lambda = \lambda(x,\xi) = \lambda_s (x,\xi) > 0$ such that 
\begin{equation*}
\lambda (x,\xi)^{-2} | x |^2 + \lambda (x,\xi)^{-2s} | \xi |^2 = 1. 
\end{equation*}
Then $(x,\xi) \in \sr {2d-1}$ if and only if $\lambda (x,\xi) = 1$. 
By the implicit function theorem the function $\lambda: \rr {2d} \setminus 0 \to \ro_+$ is smooth \cite{Krantz1}. 

If $\mu > 0$ and $(x,\xi) \in \sr {2d-1}$ then $\lambda ( \mu x, \mu^s \xi) = \mu = \mu \lambda (x,\xi)$. 
In fact 
\begin{equation}\label{eq:quasihomogen1}
\lambda ( \mu x, \mu^s \xi) = \mu \lambda (x,\xi)
\end{equation}
holds for any $(x,\xi) \in \rr {2d} \setminus 0$ and $\mu > 0$ by the following argument. 
Given $(x,\xi) \in \rr {2d} \setminus 0$ set $\mu_1 = \lambda (x,\xi)$ so that 
$(x/ \mu_1, \xi/\mu_1^s) \in \sr {2d-1}$. 
Then for $\mu > 0$
\begin{equation*}
\lambda ( \mu x, \mu^s \xi) = \lambda ( \mu \mu_1 x/\mu_1, (\mu \mu_1)^s \xi/\mu_1^s ) 
= \mu \mu_1 
= \mu \lambda (x,\xi). 
\end{equation*}

We may define the projection $p(x,\xi) = p_s(x,\xi)$ of $(x,\xi) \in \rr {2d} \setminus 0$ along the curve $\ro_+ \ni \mu \mapsto (\mu x, \mu^s \xi)$ onto $\sr {2d-1}$. 
This means 
\begin{equation}\label{eq:projection}
p(x,\xi) = \left( \lambda(x,\xi)^{-1} x, \lambda(x,\xi)^{-s} \xi \right), \quad (x,\xi) \in \rr {2d} \setminus 0. 
\end{equation}
Due to \eqref{eq:quasihomogen1} $p(\mu x, \mu^s \xi) = p(x, \xi)$ does not depend on $\mu > 0$. 
The function $p: \rr {2d} \setminus 0 \to \sr {2d-1}$ is smooth since $\lambda \in C^\infty(\rr {2d} \setminus 0)$
and $\lambda(x,\xi) > 0$ for all $(x,\xi) \in \rr {2d} \setminus 0$. 

From \cite{Parenti1}, or by straightforward arguments, we have the bounds
\begin{equation}\label{eq:lambdaboundnonisotropic}
|x| + |\xi|^{\frac1s}
\lesssim \lambda(x,\xi) \lesssim |x| + |\xi|^{\frac1s}, \quad (x,\xi) \in \rr {2d} \setminus 0
\end{equation}
and
\begin{equation}\label{eq:lambdaboundisotropic}
\eabs{ (x,\xi) }^{\min \left( 1, \frac1s \right)}
\lesssim 1 + \lambda(x,\xi) 
\lesssim \eabs{(x,\xi)}^{\max \left( 1, \frac1s \right)}, \quad (x,\xi) \in \rr {2d} \setminus 0. 
\end{equation}

H\"ormander type symbol classes with anisotropic behavior in the frequency domain can be found in \cite[D\'efinition~1.3]{Lascar1} and in \cite[Definition~1.4]{Parenti1}. 
Now we define symbol classes that are adaptations of this concept to the Shubin calculus. 

\begin{defn}\label{def:symbol}
Let $s > 0$ and $m \in \ro$. 
The space of ($s$-)anisotropic Shubin symbols $G^{m,s}$ of order $m$ consists of functions $a \in C^\infty(\rr {2d})$ 
that satisfy the estimates
\begin{equation*}
|\pdd x \alpha \pdd \xi \beta a(x,\xi)|
\lesssim ( 1 + |x| + |\xi|^{\frac1s} )^{m - |\alpha| - s |\beta|}, \quad (x,\xi) \in T^* \rr d, \quad \alpha, \beta \in \nn d. 
\end{equation*}
\end{defn}

The symbols $G^{m,s}$ enjoy the following symmetry: 
If $b(x,\xi) = a(\xi,x)$ then $a \in G^{m,s}$ if and only if $b \in G^{m/s,1/s}$. 
It is clear that  
\begin{equation*}
\bigcap_{m \in \ro} G^{m,s} = \cS(\rr {2d}). 
\end{equation*}
Referring to the weight \eqref{eq:weightanisotrop} we use the seminorms on $a \in G^{m,s}$ indexed by $j \in \no$
\begin{equation}\label{eq:seminormGms}
\| a \|_j = \max_{|\alpha + \beta| \leqs j}
\sup_{(x,\xi) \in \rr {2d} } \mu_s(x,\xi)^{-m + |\alpha| + s |\beta|} \left| \pdd x \alpha \pdd \xi \beta a(x,\xi ) \right|. 
\end{equation}

The symbol classes $G^{m,s}$ with $s \in \qo_+$ (positive rationals) were introduced in \cite[Definition~3.1]{Parenti1}
as a tool in order to construct parametrices for pseudodifferential operators. 
Here we generalize to $s \in \ro_+$ which is a straightforward extension concerning the calculus. 
In \cite[Section~3]{Parenti1} results for a calculus for the symbol classes $G^{m,s}$ are briefly stated without proofs. 
In this section we prove in detail the basic calculus results for the anisotropic Shubin symbols $G^{m,s}$.

We have $G^{m,1} = G^m = G_1^m$, that is the usual Shubin class, 
and we cannot embed $G_\rho^m$ in a space $G^{n,s}$ unless $\rho = s = 1$. 
Using \eqref{eq:lambdaboundnonisotropic} and \eqref{eq:lambdaboundisotropic} the embedding 
\begin{equation}\label{eq:Gmsinclusion}
G^{m,s} \subseteq G_\rho^{m_0}, 
\end{equation}
where $m_0 = \max(m, m/s)$ and $\rho = \min(s, 1/s)$, can be confirmed. 
Thus the Shubin calculus \cite{Shubin1} applies to the anisotropic Shubin symbols. 

We also note that the more general pseudodifferential calculus in \cite{Nicola1} is not directly applicable to the symbol classes $G^{m,s}$ unless $s = 1$. 
In fact if $s \neq 1$ then either the space weight function $\Phi(x,\xi) = 1 + |x| + |\xi|^{\frac1s}$ or the frequency weight function
$\Psi(x,\xi) = 1 + |x|^s + |\xi|$ is not sublinear. 
Nevertheless from \eqref{eq:lambdaboundisotropic} it follows that $G^{m,s} \subseteq S(M; \Phi, \Psi)$ 
as defined in \cite[Definition~1.1.1]{Nicola1} with 
$M(x,\xi) = \eabs{(x,\xi)}^{\max(m, m/s)}$, $\Phi(x,\xi) = \eabs{(x,\xi)}^{\min \left(1, \frac1s \right)}$
and $\Psi(x,\xi) = \eabs{(x,\xi)}^{\min(1,s)}$. 
Thus the pseudodifferential calculus in \cite[Chapter~1.2]{Nicola1} applies to $G^{m,s}$, 
but the anisotropy is again lost. 

There is a more subtle anisotropic subcalculus adapted to the anisotropic Shubin symbols $G^{m,s}$, for each fixed $s > 0$, 
which preserves the anisotropy.
We deduce a minimal such calculus and start 
with asymptotic expansions. 

Given a sequence of symbols $a_j \in G^{m_j,s}$, $j=1,2,\dots$, such that $m_j \to - \infty$ as $j \to + \infty$
we write 
\begin{equation*}
a \sim \sum_{j = 1}^\infty a_j
\end{equation*}
provided that for any $n \geqs 2$
\begin{equation*}
a - \sum_{j = 1}^{n-1} a_j \in G^{\mu_n,s}
\end{equation*}
where $\mu_n = \max_{j \geqs n} m_j$.

\begin{lem}\label{lem:asymptoticsum}
Let $s > 0$. 
Given a sequence of symbols $a_j \in G^{m_j,s}$, $j=1,2,\dots$, such that $m_j \to - \infty$ as $j \to + \infty$, 
there exists a symbol $a \in G^{m,s}$ where $m = \max_{j \geqs 1} m_j$ such that $a \sim \sum_{j = 1}^\infty a_j$. The symbol $a$ is unique modulo addition with a function in $\cS(\rr {2d})$. 
\end{lem}

\begin{proof}
Let $\fy \in C^\infty (\rr {2d})$ satisfy $0 \leqs \fy \leqs 1$, $\fy(z) = 0$ if $|z| \leqs \frac12$ and
$\fy(z) = 1$ if $|z| \geqs 1$. 
Set for $t \geqs 1$
\begin{equation*}
\psi(x,\xi) = \fy( t^{-1} x, t^{-s} \xi), \quad (x,\xi) \in T^* \rr d. 
\end{equation*}
Then for all $t \geqs 1$ we have 
\begin{equation*}
\left| \pdd x \alpha \pdd \xi \beta \psi(x,\xi) \right| 
\leqs C_{\alpha,\beta} \mu_s(x,\xi)^{-|\alpha| - s |\beta|}. 
\end{equation*}
If fact this is trivial if $\alpha = \beta = 0$. 
If instead $(\alpha,\beta) \in \nn {2d} \setminus 0$ then 
\begin{equation*}
\frac14 \leqs t^{-2} |x|^2 + t^{-2 s} |\xi|^2 \leqs 1
\end{equation*}
in the support of $\pdd x \alpha \pdd \xi \beta \fy ( t^{-1}x, t^{-s} \xi)$. 
Thus $|x| + |\xi|^{\frac1s} \lesssim t$ in said support. 
This gives
\begin{equation}\label{eq:psiestimate1}
\begin{aligned}
\left| \pdd x \alpha \pdd \xi \beta \psi(x,\xi) \right|
& = t^{-|\alpha| - s |\beta|} \left| ( \pdd x \alpha \pdd \xi \beta \fy)( t^{-1} x, t^{-s} \xi) \right| \\
& \leqs C_{\alpha,\beta} \mu_s (x,\xi)^{ -|\alpha| - s |\beta| }. 
\end{aligned}
\end{equation}

The symbol $a$ is constructed as 
\begin{equation*}
a(x,\xi) = \sum_{j = 1}^\infty \fy( t_j^{-1} x, t_j^{-s} \xi) a_j (x,\xi) 
\end{equation*}
for a sufficiently rapidly increasing sequence $(t_j) \subseteq \ro_+$. 
Given $n \geqs 2$ we must show $a - \sum_{j = 1}^{n-1} a_j \in G^{\mu_n,s}$. 
We have 
\begin{equation*}
a (x,\xi) - \sum_{j = 1}^{n-1} a_j (x,\xi)
= \sum_{j = 1}^{n-1} \left( \fy( t_j^{-1} x, t_j^{-s} \xi) - 1 \right) a_j (x,\xi) 
+ \sum_{j = n}^{\infty} \fy( t_j^{-1} x, t_j^{-s} \xi) a_j (x,\xi). 
\end{equation*}
The first sum is compactly supported and hence belongs to $G^{\mu_n,s}$ trivially
so it suffices to prove  
\begin{equation}\label{eq:remainstoprove1}
\sum_{j = n}^{\infty} \fy( t_j^{-1} x, t_j^{-s} \xi) a_j (x,\xi) \in G^{\mu_n,s}. 
\end{equation}

First we show 
\begin{equation}\label{eq:step1}
\left| \pdd x \alpha \pdd \xi \beta \left( \fy( t_j^{-1} x, t_j^{-s} \xi) a_j (x,\xi) \right) \right| 
\leqs 2^{- j} \mu_s (x,\xi)^{m_j + 1 -|\alpha| - s |\beta|}
\end{equation}
for all $j \geqs 1$ and $|\alpha + \beta| \leqs j$, 
provided $t_j > 0$ is sufficiently large. 
In fact this estimate is a consequence of $a_j \in G^{m_j,s}$,  \eqref{eq:psiestimate1}, Leibniz' rule, 
and the support properties of $\fy( t_j^{-1} x, t_j^{-s} \xi)$, if $t_j > 0$ is sufficiently large. 

Let $\alpha, \beta \in \nn d$ and pick $N \geqs \max(n+1, |\alpha + \beta| )$ such that $\mu_N \leqs \mu_n - 1$. 
Then for all $j \geqs N$ it holds $m_j \leqs \mu_j \leqs \mu_N \leqs \mu_n - 1$. 
Combined with \eqref{eq:step1} this gives
\begin{equation*}
\sum_{j = N}^{\infty} \left| \pdd x \alpha \pdd \xi \beta \left( \fy( t_j^{-1} x, t_j^{-s} \xi) a_j (x,\xi) \right) \right| 
\leqs 2^{1-N}  \mu_s (x,\xi)^{\mu_n - |\alpha| - s |\beta|}. 
\end{equation*}
Since $\sum_{j = n}^{N-1} \fy( t_j^{-1} x, t_j^{-s} \xi) a_j (x,\xi) \in G^{\mu_n,s}$
we have proved \eqref{eq:remainstoprove1}. 
\end{proof}

We have the following asymptotic expansion for the Weyl product of $a \in G^{m,s}$ and $b \in G^{n,s}$, $m,n \in \ro$ \cite{Shubin1}:
\begin{equation}\label{eq:calculuscomposition1}
a \wpr b(x,\xi) \sim \sum_{\alpha, \beta \geqs 0} \frac{(-1)^{|\beta|}}{\alpha! \beta!} \ 2^{-|\alpha+\beta|}
D_x^\beta \pdd \xi \alpha a(x,\xi) \, D_x^\alpha \pdd \xi \beta b(x,\xi). 
\end{equation}

Each term in the sum belongs to $G^{m+n -(1+s)|\alpha+\beta|,s}$.

In the next result we show that the symbol classes $G^{m,s}$ are invariant with respect to 
the parameter $t \in \ro$ in \eqref{eq:tquantization}. 
In other words if one changes quantization one gets a new symbol in the same class. 
Combined with $a^w(x,D)^* = \overline{a}^w(x,D)$, an immediate consequence 
is that for each $t \in \ro$ the symbol class $G^{m,s}$ is closed with respect to formal adjoint: 
If $a_t \in G^{m,s}$ and $a_t(x,D)^*  = b_t(x,D)$ then $b_t \in G^{m,s}$. 

We also show the continuity of the bilinear Weyl product on the symbol classes $G^{m,s}$. 
Again by the first result the continuity extends to the symbol product in the $t$-quantization 
for any $t \in \ro$.

\begin{prop}\label{prop:calculusanisotropic}
Let $s > 0$ and $m,n \in \ro$. 

\begin{enumerate}[(i)]

\item If $t \in \ro$ and $a \in G^{m,s}$ then $b(x,\xi) = e^{i t \la D_x, D_\xi \ra} a (x, \xi) \in G^{m,s}$, 
and the map $a \mapsto b$ is continuous on $G^{m,s}$. 

\item If $a \in G^{m,s}$ and $b \in G^{n,s}$ then $a \wpr b \in G^{m+n,s}$, and the Weyl product is continuous
\begin{equation*}
\wpr: G^{m,s} \times G^{n,s} \to G^{m+n,s}. 
\end{equation*}

\end{enumerate}

\end{prop}

\begin{proof}

\textit{(i)}
We may assume $t \neq 0$ since the claim is trivial otherwise. 
Let $\alpha, \beta \in \nn d$. 
The operator $e^{i t \la D_x, D_\xi \ra}$ commutes with differential operators $\pdd x \alpha \pdd \xi \beta$. 
The distribution $\pdd x \alpha \pdd \xi \beta b = \cF^{-1} \left( e^{i t \la \cdot, \cdot \ra} \wh { \pdd x \alpha \pdd \xi \beta a } \right)$ is well defined in $\cS'(\rr {2d})$. 

Let $\chi \in C_c^\infty(\rr {2d})$ satisfy $0 \leqs \chi \leqs 1$, 
$\chi(z) = 1$ when $|z| \leqs 1$ and $\chi(z) = 0$ when $|z| \geqs 2$.
Set $\chi_\ep (z) = \chi( \ep z)$ for $\ep > 0$. 
Then $\chi_\ep (\pdd x \alpha \pdd \xi \beta a) \to \pdd x \alpha \pdd \xi \beta a$ in $\cS'(\rr {2d})$ as $\ep \to 0^+$.
Hence we obtain from \eqref{eq:quantizationchange} 
\begin{align*}
\pdd x \alpha \pdd \xi \beta b(x,\xi)
& = \cF^{-1} \left( e^{i t \la \cdot, \cdot \ra} \wh {\pdd x \alpha \pdd \xi \beta a} \right) (x,\xi)
= \lim_{\ep \to 0^+} \cF^{-1} \left( e^{i t \la \cdot, \cdot \ra} \cF \left( \chi_\ep \pdd x \alpha \pdd \xi \beta a \right) \right) (x,\xi) \\
& = (2 \pi | t |)^{-d} \lim_{\ep \to 0^+} \int_{\rr {2d}} e^{ - \frac{i}{t} \la x-y,\xi-\eta \ra}  \chi_\ep (y,\eta) \pdd x \alpha \pdd \xi \beta a(y,\eta )  \, \dd y \, \dd \eta 
\end{align*}
in $\cS'(\rr {2d})$. 

Define the operator
\begin{equation*}
(S f) (y,\eta) = (1 - \Delta_{y,\eta} ) \left( \eabs{ t^{-1} (x-y,\xi-\eta)}^{-2} f (y,\eta) \right) 
\end{equation*}
acting on $f \in C^\infty (\rr {2d})$. 
From
\begin{equation*}
(1 - \Delta_{y,\eta} ) e^{ - \frac{i}{t} \la x-y,\xi-\eta \ra}
= \eabs{ t^{-1} (x-y,\xi-\eta)}^2 e^{ - \frac{i}{t} \la x-y,\xi-\eta \ra}
\end{equation*}
we obtain from integration by parts for $N \in \no$
\begin{align*}
(2 \pi | t |)^{d} \pdd x \alpha \pdd \xi \beta b(x,\xi)
& = \lim_{\ep \to 0^+} \int_{\rr {2d}} e^{ - \frac{i}{t} \la x-y,\xi-\eta \ra} S^N \left(  \chi_\ep (y,\eta) \pdd x \alpha \pdd \xi \beta a(y,\eta ) \right) \dd y \, \dd \eta \\
& = \int_{\rr {2d}} e^{ - \frac{i}{t} \la x-y,\xi-\eta \ra} S^N \left( \pdd x \alpha \pdd \xi \beta a(y,\eta ) \right) \dd y \, \dd \eta
\end{align*}
by dominated convergence, since $S^N \pdd x \alpha \pdd \xi \beta a \in L^1(\rr {2d})$ provided $N$ is large enough.

This gives using \eqref{eq:lambdaboundisotropic}, \eqref{eq:seminormGms} and Lemma \ref{lem:speetre} 
\begin{align*}
| \pdd x \alpha \pdd \xi \beta b(x,\xi)|
& \lesssim \int_{\rr {2d}} \left| S^N \left(  \pdd x \alpha \pdd \xi \beta a(y,\eta ) \right) \right| \, \dd y \, \dd \eta \\
& \leqs C_{t,N} \| a \|_{2N + |\alpha+\beta|} \int_{\rr {2d}}  \eabs{ (x-y,\xi-\eta)}^{-2N} \mu_s(y,\eta)^{ m- |\alpha| - s |\beta| }  \, \dd y \, \dd \eta \\
& = C_{t,N} \| a \|_{2N + |\alpha+\beta|} \int_{\rr {2d}}  \eabs{ (y,\eta)}^{-2N} \mu_s( x - y,\xi - \eta)^{ m- |\alpha| - s |\beta| } \, \dd y \, \dd \eta \\
& \lesssim C_{t,N} \| a \|_{2N + |\alpha+\beta|} \mu_s( x,\xi)^{ m- |\alpha| - s |\beta| } \int_{\rr {2d}}  \eabs{ (y,\eta)}^{-2N} \mu_s( y,\eta)^{|m| + |\alpha| + s |\beta| } \, \dd y \, \dd \eta \\
& \leqs C_{t,N} \| a \|_{2N + |\alpha+\beta|} \mu_s( x,\xi)^{ m- |\alpha| - s |\beta| } \int_{\rr {2d}}  \eabs{ (y,\eta)}^{-2N + ( |m| + |\alpha| + s |\beta| )\max \left( 1, \frac1s \right)} \, \dd y \, \dd \eta \\
& \leqs C_{t,N} \| a \|_{2N + |\alpha+\beta|} \mu_s( x,\xi)^{ m- |\alpha| - s |\beta| }
\end{align*}
after possibly increasing $N$ (which may depend on $|\alpha+\beta|$). 
In view of \eqref{eq:seminormGms} we obtain for any $j \in \no$
\begin{equation*}
\| b \|_{j}
\leqs C_{t,N} \| a \|_{2N_j+ j} 
\end{equation*}
for some $N_j \in \no$,
which proves claim \textit{(i)}.

\textit{(ii)}
Due to \eqref{eq:Gmsinclusion} we may use results for the calculus of Shubin symbols $G_\rho^m$. 

When $a,b \in \cS(\rr {2d})$ we have by \eqref{eq:weylproduct1} $a \wpr b(z) = f(z,z)$ where 
\begin{equation*}
f (z,w) = e^{\frac{i}{2} \sigma (D_z, D_w)} (a \otimes b )(z,w), \quad z,w \in \rr {2d}.  
\end{equation*}
Suppose $a \in G^{m,s}$ and $b \in G^{n,s}$. 
Set $a_\ep = \chi_\ep a$ and $b_\ep = \chi_\ep b$ where  $\chi \in C_c^\infty(\rr {2d})$ and $\chi_\ep$ is defined as above. 
Then $a_\ep \otimes b_\ep \to a \otimes b$ in $\cS'(\rr {4d})$ as $\ep \to 0^+$. 
Since $e^{\frac{i}{2} \sigma (D_z, D_w)}$ is continuous on $\cS'(\rr {4d})$ it follows that
\begin{equation}\label{eq:weylproductgen1}
f  (z,w) = \lim_{\ep \to 0^+} e^{\frac{i}{2} \sigma (D_z, D_w)} (a_\ep \otimes b_\ep ) (z,w)
\end{equation}
in $\cS'(\rr {4d})$. 

From the argument in the proof of \cite[Theorem~A.5]{Schulz1} it follows that the limit \eqref{eq:weylproductgen1} is actually
pointwise for all $z,w \in \rr {2d}$. 
The Fourier multiplier operator $e^{\frac{i}{2} \sigma (D_z, D_w)}$ commutes with differential operators so 
for any $\alpha,\beta \in \nn {2d}$ we have the pointwise limit
\begin{equation}\label{eq:weylproductgen2}
\pdd z \alpha \pdd w \beta f  (z,w) = \lim_{\ep \to 0^+} e^{\frac{i}{2} \sigma (D_z, D_w)} (\pd \alpha a_\ep \otimes \pd \beta b_\ep ) (z,w)
\end{equation}
which yields using \eqref{eq:weylproduct2} 
\begin{equation}\label{eq:weylproductgen3}
\begin{aligned}
\pd \alpha (a \wpr b)(z)
& = \pd \alpha ( f(z,z) )
= \sum_{\beta \leqs \alpha} \binom{\alpha}{\beta} (\pdd z \beta \pdd w {\alpha - \beta} f)  (z,z) \\
& = \sum_{\beta \leqs \alpha} \binom{\alpha}{\beta} \lim_{\ep \to 0^+} e^{\frac{i}{2} \sigma (D_z, D_w)} (\pd \beta a_\ep \otimes \partial^{\alpha-\beta} b_\ep ) (z,z) \\
& = \pi^{-2d} \sum_{\beta \leqs \alpha} \binom{\alpha}{\beta} 
\lim_{\ep \to 0^+} 
\iint_{\rr {4d}} e^{2 i \sigma (z-v, z-u)}  \pd \beta a_\ep(u) \partial^{\alpha-\beta} b_\ep (v) \, \dd u \, \dd v. 
\end{aligned}
\end{equation}

Next we note 
\begin{equation*}
(1 - \Delta_{u,v} ) e^{2 i \sigma (z-v, z-u)}
= \eabs{ 2(z-u, z- v)}^2 e^{2 i \sigma (z-v, z-u)}. 
\end{equation*}

If we define the operator
\begin{equation*}
(S f) (u,v) = (1 - \Delta_{u,v} ) \left( \eabs{ 2(z-u, z- v)}^{-2} f (u,v) \right), \quad u,v \in \rr {2d},  
\end{equation*}
acting on $f \in C^\infty (\rr {4d})$, then we obtain for $N \in \no$ 
using integration by parts and dominated convergence
\begin{equation*}
\begin{aligned}
& \lim_{\ep \to 0^+} 
\iint_{\rr {4d}} e^{2 i \sigma (z-v, z-u)}  \pd \beta a_\ep(u) \partial^{\alpha-\beta} b_\ep (v) \, \dd u \, \dd v \\
& = 
\lim_{\ep \to 0^+} 
\iint_{\rr {4d}} e^{2 i \sigma (z-v, z-u)} S^N \left(  \pd \beta a_\ep(u) \partial^{\alpha-\beta} b_\ep (v) \right) \, \dd u \, \dd v \\
& = 
\iint_{\rr {4d}} e^{2 i \sigma (z-v, z-u)} S^N \left(  \pd \beta a(u) \partial^{\alpha-\beta} b (v) \right) \, \dd u \, \dd v 
\end{aligned}
\end{equation*}
since $S^N \left(  \pd \beta a \otimes \partial^{\alpha-\beta} b \right) \in L^1(\rr {4d})$ provided $N$ is sufficiently large.  

We denote $\alpha = (\alpha_1, \alpha_2) \in \nn {2d}$ with $\alpha_1,\alpha_2 \in \nn d$. 
Combining with \eqref{eq:weylproductgen3} and using 
\eqref{eq:lambdaboundisotropic}, \eqref{eq:seminormGms} and Lemma \ref{lem:speetre} we obtain 
\begin{equation*}
\begin{aligned}
& \left| \pd \alpha (a \wpr b)(z) \right| \\
& \lesssim \sum_{\beta \leqs \alpha} \binom{\alpha}{\beta} 
\iint_{\rr {4d}} \left| S^N  \left(  \pd \beta a(u) \partial^{\alpha-\beta} b (v) \right) \right| \dd u \, \dd v \\
& \lesssim \sum_{\beta \leqs \alpha} \binom{\alpha}{\beta} 
\| a \|_{2N + |\beta|} \| b \|_{2N + |\alpha-\beta|}  \\
& \qquad \times \iint_{\rr {4d}} \eabs{ (z-u, z- v) }^{-2N}  \mu_s( u )^{m-|\beta_1| - s |\beta_2|} \mu_s( v)^{n -|\alpha_1-\beta_1| - s |\alpha_2-\beta_2|}  \dd u \, \dd v \\
& \leqs
\| a \|_{2N + |\alpha|} \| b \|_{2N + |\alpha|} \\
& \quad \times \sum_{\beta \leqs \alpha} \binom{\alpha}{\beta}  \iint_{\rr {4d}} \eabs{ (u, v) }^{-2N}  \mu_s( z- u )^{m-|\beta_1| - s |\beta_2|} \mu_s( z-v)^{n -|\alpha_1-\beta_1| - s |\alpha_2-\beta_2|}  \dd u \, \dd v \\
& \lesssim
\| a \|_{2N + |\alpha|} \| b \|_{2N + |\alpha|} \mu_s(z)^{m + n - |\alpha_1| - s |\alpha_2|} \\
& \quad \times \sum_{\beta \leqs \alpha} \binom{\alpha}{\beta}  \iint_{\rr {4d}} \eabs{ (u, v) }^{-2N + (|m| + |n| + 2 |\alpha_1| + 2 s |\alpha_2|) \max \left(1, \frac1s \right)} \dd u \, \dd v \\
& \lesssim
\| a \|_{2N + |\alpha|} \| b \|_{2N + |\alpha|} \mu_s(z)^{m + n - |\alpha_1| - s |\alpha_2|} 
\end{aligned}
\end{equation*}
if $N$ is sufficiently large. This shows that for any $\alpha \in \nn {2d}$ we have 
\begin{equation*}
\sup_{z \in \rr {2d}} \mu_s(z)^{-m - n + |\alpha_1| + s |\alpha_2|} \left| \pd \alpha (a \wpr b)(z) \right| 
\lesssim \| a \|_{2N + |\alpha|} \| b \|_{2N + |\alpha|}  
\end{equation*}
and the claimed continuity follows in view of \eqref{eq:seminormGms}. 
\end{proof}

\subsection{$s$-conic cutoff functions}

A family of open $s$-conic subsets are defined and denoted as follows. Recall the projection function \eqref{eq:projection} 
$p: \rr {2d} \setminus 0 \to \sr {2d-1}$. 

\begin{defn}\label{def:scone1}
Suppose $s, \ep > 0$ and $z_0 \in \sr {2d-1}$. 
Then
\begin{equation*}
\Gamma_{s, z_0, \ep}
= \{ (x,\xi) \in \rr {2d} \setminus 0, \ | z_0 - p(x,\xi) | < \ep \}
\subseteq  T^* \rr d \setminus 0. 
\end{equation*}
\end{defn}

For simplicity we write $\Gamma_{z_0, \ep} = \Gamma_{s, z_0, \ep}$ when $s$ is fixed and understood from the context. 
If $\ep > 2$ then $\Gamma_{z_0, \ep} = T^* \rr d \setminus 0$ so we usually restrict to $\ep \leqs 2$.

Next we construct cutoff functions $\chi \in G^{0,s}$ such that 
$0 \leqs \chi \leqs 1$,
$\supp \chi \subseteq \Gamma_{z_0, 2\ep} \setminus \rB_{r/2}$, 
$\chi |_{\Gamma_{z_0, \ep} \setminus \overline \rB_{r} } \equiv 1$
for given $\ep , r > 0$, and $z_0 \in \sr {2d-1}$. 
They will be needed in Section \ref{sec:microellipticgabor}. 

\begin{lem}\label{lem:cutoff}
Let $s > 0$. 
If $r > 0$, $0 < \ep \leqs 1$ and $z_0 \in \sr {2d-1}$
then there exists $\chi \in G^{0,s}$ such that 
$0 \leqs \chi \leqs 1$,
$\supp \chi \subseteq \Gamma_{z_0, 2\ep} \setminus \rB_{r/2}$
and $\chi |_{\Gamma_{z_0, \ep} \setminus \overline \rB_{r} } \equiv 1$. 
\end{lem}

\begin{proof}
Let $\fy \in C_c^\infty (\rr {2d})$ satisfy 
$0 \leqs \fy \leqs 1$, 
$\supp \fy \subseteq z_0 + \rB_{2 \ep}$ 
and $\fy |_{z_0 + \rB_{\ep}} \equiv 1$. 
Let $g \in C^\infty(\ro)$ satisfy $0 \leqs g \leqs 1$, $g(x) = 0$ if $x \leqs \frac12$ and $g(x) = 1$ if $x \geqs 1$.  
Set 
\begin{equation}\label{eq:homogen1}
\psi (\lambda x, \lambda^s \xi) = \fy (x,\xi), \quad (x,\xi) \in \sr {2d-1}, \quad \lambda > 0, 
\end{equation}
and 
\begin{equation}\label{eq:cutoff1}
\chi (z) = g ( r^{-1} |z| ) \psi (z), \quad z \in \rr {2d}. 
\end{equation}

Note that \eqref{eq:homogen1} can be written
\begin{equation*}
\psi (x,\xi) = \fy ( p(x,\xi) ), \quad (x,\xi) \in \rr {2d} \setminus 0, 
\end{equation*}
and it follows that $\psi \in C^\infty( \rr {2d} \setminus 0 )$, 
and thus 
$\chi \in C^\infty (\rr {2d})$. 
The properties 
$\chi |_{\Gamma_{z_0, \ep} \setminus \overline \rB_{r} } \equiv 1$ and 
$\supp \chi \subseteq \Gamma_{z_0, 2\ep} \setminus \rB_{r/2}$ follow. 

From \eqref{eq:homogen1} we obtain
\begin{equation}\label{eq:homogen2}
\pdd x \alpha \pdd \xi \beta \fy (x,\xi)
= \lambda^{|\alpha| + s |\beta|} (\pdd x \alpha \pdd \xi \beta \psi) (\lambda x, \lambda^s \xi), 
\quad (x,\xi) \in \sr {2d-1}, \quad \lambda > 0. 
\end{equation}

Let $(y,\eta) \in \rr {2d}$ satisfy $|(y,\eta)| > \frac{r}{2}$. 
Then $(y,\eta) = (\lambda x, \lambda^s \xi)$ for a unique $(x,\xi) \in \sr {2d-1}$ and a unique 
\begin{equation*}
\lambda > \delta := \min \left(  \frac{r}{ 2 }, \left( \frac{r}{ 2} \right)^{\frac1s} \right) > 0.
\end{equation*}
We have 
\begin{equation*}
1 + |y| + |\eta|^{\frac1s} = 1 + \lambda ( |x| + |\xi|^{\frac1s} ) \leqs 2 (1+ \lambda). 
\end{equation*}
Thus we obtain from \eqref{eq:homogen2} for any $\alpha, \beta \in \nn d$
\begin{equation*}
\left| \pdd y \alpha \pdd \eta \beta \psi (y, \eta) \right|
\leqs C_{\alpha,\beta} (1+\lambda)^{-|\alpha| - s |\beta|} 
\lesssim ( 1 + |y| + |\eta|^{\frac1s} )^{-|\alpha| - s |\beta|}. 
\end{equation*}
From \eqref{eq:cutoff1} we may conclude that $\chi \in G^{0,s}$. 
\end{proof}

Sometimes it is useful to have the following alternative to the $s$-conic neighborhoods of Definition \ref{def:scone1}. 

\begin{defn}\label{def:scone2}
Suppose $s, \ep > 0$ and $(x_0, \xi_0) \in \sr {2d-1}$. 
Then
\begin{align*}
\wt \Gamma_{s, (x_0, \xi_0), \ep}
= \wt \Gamma_{ (x_0, \xi_0), \ep}
& = \{ (x,\xi) \in \rr {2d} \setminus 0: \ (x,\xi) = (\lambda (x_0 + y), \lambda^s (\xi_0 + \eta), \ \lambda > 0, \ (y,\eta) \in \rB_\ep \} \\
& = \{ (x,\xi) \in \rr {2d} \setminus 0: \ \exists \lambda > 0 : (\lambda x, \lambda^s \xi) \in (x_0,\xi_0) + \rB_\ep \} 
\subseteq  T^* \rr d \setminus 0. 
\end{align*}
\end{defn}

Again $\wt \Gamma_{ (x_0, \xi_0), \ep}$ is $s$-conic. 

The neighborhoods $\Gamma_{s, (x_0, \xi_0), \ep}$ and $\wt \Gamma_{s, (x_0, \xi_0), \ep}$
are not identical, even if $s = 1$ in which case $p(x,\xi) = (x,\xi)/| (x,\xi) |$. 
But by the following result the $s$-conic neighborhoods of the form $\Gamma_{s,z_0,\ep}$ and $\wt \Gamma_{s, z_0, \ep}$
are equivalent topologically.

\begin{lem}\label{lem:sconesequiv}
Let $z_0 \in \sr {2d-1}$. 
For each $\ep > 0$ there exists $\delta > 0$ such that 
\begin{equation}\label{eq:sconeinclusion1}
\Gamma_{z_0, \delta} \subseteq \widetilde \Gamma_{z_0,\ep}
\end{equation}
and
\begin{equation}\label{eq:sconeinclusion2}
\wt \Gamma_{z_0, \delta} \subseteq \Gamma_{z_0,\ep}. 
\end{equation}
\end{lem}

\begin{proof}
Let $z_0 = (x_0, \xi_0)$. 
If $\ep > 0$ and $(x,\xi) \in \Gamma_{z_0, \ep} \cap \sr {2d-1}$ then $(x,\xi) \in (x_0,\xi_0) + \rB_\ep$ so $(x,\xi) \in \widetilde \Gamma_{z_0,\ep}$. 
Since both $\Gamma_{z_0, \ep}$ and $\widetilde \Gamma_{z_0, \ep}$ are $s$-conic, this shows 
\begin{equation*}
\Gamma_{z_0, \ep} \subseteq \widetilde \Gamma_{z_0,\ep} 
\end{equation*}
for any $\ep > 0$. Thus \eqref{eq:sconeinclusion1} follows with $\delta = \ep$. 

In order to show \eqref{eq:sconeinclusion2} let $\ep > 0$, 
and suppose $0 < \delta < 1$. 
If $(x,\xi) \in \widetilde \Gamma_{z_0, \delta} \cap \sr {2d-1}$
then there exists $\mu = \mu(x,\xi) > 0$ such that $| (\mu x, \mu^s \xi) - (x_0, \xi_0) | < \delta$. 
We have 
\begin{align*}
\min (\mu, \mu^s) & \leqs |(\mu x, \mu^s \xi)| < 1 + \delta, \\
\max (\mu, \mu^s) & \geqs |(\mu x, \mu^s \xi)| > 1 - \delta
\end{align*}
which gives
\begin{equation*}
(1-\delta)^{\max \left( 1, \frac1s \right)} < \mu (x,\xi) < (1+\delta)^{\max \left( 1, \frac1s \right)} 
\quad \forall (x,\xi) \in \widetilde \Gamma_{z_0, \delta} \cap \sr {2d-1}. 
\end{equation*}
Thus we may pick $\delta < \ep/2$ such that 
\begin{equation*}
\max \left( \left| 1- \mu(x,\xi) \right|, \left| 1 - \mu(x,\xi)^s \right| \right)< \ep/2 
\quad \forall (x,\xi) \in \widetilde \Gamma_{z_0, \delta} \cap \sr {2d-1}. 
\end{equation*}

If $(x,\xi) \in \widetilde \Gamma_{z_0, \delta} \cap \sr {2d-1}$ then
$p(x,\xi) = (x,\xi)$ so we obtain
\begin{align*}
| p(x,\xi) - (x_0,\xi_0) |
& = \left| (\mu (x,\xi) x, \mu(x,\xi)^s \xi)  - (x_0,\xi_0) + \left( ( 1-\mu(x,\xi) )x,(1-\mu(x,\xi)^s ) \xi \right) \right| \\
& < \delta + \max \left( |1-\mu(x,\xi)|,  |1-\mu(x,\xi)^s | \right) < \ep. 
\end{align*}
Again due to $s$-conic property of $\widetilde \Gamma_{z_0, \delta}$ and $\Gamma_{z_0,\ep}$, 
this shows $\widetilde \Gamma_{z_0, \delta} \subseteq \Gamma_{z_0,\ep}$, that is \eqref{eq:sconeinclusion2}. 
\end{proof}

In Example \ref{ex:polynomial} and in Section \ref{sec:microellipticgabor} we will use the following definition which is a natural anisotropic microlocal version of \cite[Definition~25.1]{Shubin1} as well as of \cite[Eq.~(1.11)]{Cappiello0} (cf. \cite{Cappiello4}).

\begin{defn}\label{def:noncharacteristic}
Let $s > 0$, $z_0 \in \rr {2d} \setminus 0$, and $a \in G^{m,s}$. 
Then $z_0$ is called non-characteristic of order $m_1 \leqs m$, $z_0 \notin \charac_{s,m_1} (a)$, if there exists $\ep > 0$ such that, 
with $\Gamma = \Gamma_{s,p(z_0),\ep}$,
\begin{align}
|a( x, \xi )| & \geqs C \mu_s(x,\xi)^{m_1}, \quad (x,\xi) \in \Gamma \quad, \quad |x| + |\xi|^{\frac1s} \geqs R, \label{eq:lowerbound1} \\
|\pdd x \alpha \pdd \xi \beta a(x,\xi)| &\lesssim |a(x,\xi)| \mu_s(x,\xi)^{- |\alpha| - s |\beta|}, \quad \alpha, \beta \in \nn d, \quad (x,\xi) \in \Gamma, \quad |x| + |\xi|^{\frac1s} \geqs R, \label{eq:boundderivative1}
\end{align}
for suitable $C, R > 0$. 
\end{defn}

If $m_1 = m$ we write $\charac_{s,m} (a) = \charac_{s} (a)$, and then the condition \eqref{eq:boundderivative1} is then redundant. 
Note that $\charac_{s,m_1} (a)$ is a closed $s$-conic subset of $T^* \rr d \setminus 0$, 
and $\charac_{s,m_1} (a) \subseteq \charac_{s,m_2} (a)$ if $m_1 \leqs m_2 \leqs m$. 

\begin{example}\label{ex:polynomial}
In \cite{Cappiello0,Cappiello4} polynomial symbols of the form 
\begin{equation}\label{eq:polysymbol1}
a(x,\xi) = \sum_{\frac{|\alpha|}{k} + \frac{|\beta|}{m} \leqs 1} c_{\alpha \beta} x^\alpha \xi^\beta, \quad x, \xi \in \rr d, \quad c_{\alpha \beta} \in \co, 
\end{equation}
are studied for $k, m \in \no$.  
Then $a \in G^{\max(k,m)}$ and $a \in G^{k,\frac{k}{m}}$. 
In fact we have for $(x,\xi) \in \sr {2d-1}$ and $\lambda > 0$
\begin{equation*}
\left( \pdd x \gamma \pdd \xi \kappa a \right)(\lambda x, \lambda^{\frac{k}{m}} \xi) 
= \sum_{\frac{|\alpha|}{k} + \frac{|\beta|}{m} \leqs 1} c_{\alpha \beta \gamma \kappa} 
\lambda^{|\alpha-\gamma| + \frac{k}{m} |\beta-\kappa|} x^{\alpha-\gamma} \xi^{\beta - \kappa}. 
\end{equation*}
If $(y,\eta) \in \rr {2d}$ and $|(y,\eta)| \geqs 1$ then we write $(y, \eta)  = (\lambda x, \lambda^{\frac{k}{m}} \xi)$
for $(x,\xi) \in \sr {2d-1}$ and $\lambda \geqs 1$. 
Since 
\begin{equation*}
|y| + |\eta|^{\frac{m}{k}} = \lambda \left( |x| + |\xi|^{\frac{m}{k}}  \right)
\asymp \lambda
\end{equation*}
we obtain 
\begin{align*}
\left| \pdd x \gamma \pdd \xi \kappa a (y, \eta) \right| 
& \lesssim \sum_{\frac{|\alpha|}{k} + \frac{|\beta|}{m} \leqs 1} 
( 1 + |y| + |\eta|^{\frac{m}{k}} )^{k \left( \frac{|\alpha|}{k} + \frac{|\beta|}{m} \right) - |\gamma| - \frac{k}{m} |\kappa|} \\
& \lesssim ( 1 + |y| + |\eta|^{\frac{m}{k}} )^{k - |\gamma| - \frac{k}{m} |\kappa|} 
\end{align*}
which proves that $a \in G^{k,\frac{k}{m}}$.

In \cite[Eq.~(1.11)]{Cappiello0} the symbol $a$ given by \eqref{eq:polysymbol1} is called $(k,m)$-globally elliptic if 
\begin{equation*}
\left| a(x,\xi)\right|
\geqs C \left( |x| + |\xi|^{\frac{m}{k} } \right)^k, \quad |x| + |\xi|^{\frac{m}{k}} \geqs R 
\end{equation*}
for some $C,R > 0$. 
Thus Definition \ref{def:noncharacteristic} can  be viewed as a microlocalization of $(k,m)$-global ellipticity. 
A $(k,m)$-globally elliptic symbol as above satisfies $\charac_{k/m} (a)  = \charac_{k/m, k} (a) = \emptyset$. 
\end{example}

\section{Anisotropic Gabor wave front sets}\label{sec:anisotropicgaborWF}

The following definition is inspired by H.~Zhu's \cite[Definition~1.5]{Zhu1} of a quasi-homogen-eous 
wave front set defined by two non-negative parameters. 
Zhu uses a semiclassical formulation whereas we use the STFT. 
As far as we know it is an open question to determine if the concepts coincide.

Given positive parameters $t,s > 0$
we define the $t,s$-Gabor wave front set $\WF_{\rm g}^{t,s} ( u ) \subseteq T^* \rr d \setminus 0$ of $u \in \cS'(\rr d)$. 

\begin{defn}\label{def:WFgs}
Suppose $u \in \cS'(\rr d)$, $\fy \in \cS(\rr d) \setminus 0$, and $t,s > 0$. 
A point $z_0 = (x_0,\xi_0) \in T^* \rr d \setminus 0$ satisfies $z_0 \notin \WF_{\rm g}^{t,s} ( u )$
if there exists an open set $U \subseteq T^* \rr d$ such that $z_0 \in U$ and 
\begin{equation}\label{eq:WFgs}
\sup_{(x,\xi) \in U, \ \lambda > 0} \lambda^N |V_\fy u (\lambda^t x, \lambda^s \xi)| < + \infty \quad \forall N \geqs 0. 
\end{equation}
\end{defn}

If $s = t$ we have $\WF_{\rm g}^{t,t} ( u ) = \WFg (u)$ 
which denotes the usual Gabor wave front set \cite{Hormander1,Rodino2}. 
In the definition of $\WF_g^{t,s} ( u )$ only the fraction $s/t$ matters. 
Therefore we may assume in the sequel that $t = 1$, and we write
$\WF_{\rm g}^{1,s} ( u ) = \WF_{\rm g}^s ( u )$ for simplicity. 
We call $\WF_{\rm g}^s ( u )$ the anisotropic $s$-Gabor wave front set. 
It is clear that $\WF_{\rm g}^s ( u )$ is $s$-conic. 

Referring to \eqref{eq:STFTtempered} and \eqref{eq:STFTschwartz} we see
that $\WF_{\rm g}^s ( u )$ records $s$-conic curves 
$0 < \lambda \mapsto (\lambda x, \lambda^s \xi)$ where $V_\fy u$ does not behave like the STFT of a Schwartz function. 
From \eqref{eq:STFTtempered} it also follows that it suffices to check \eqref{eq:WFgs} for $\lambda \geqs L$
where $L > 0$ may be arbitrarily large.

From \eqref{eq:STFTschwartz} it follows that $\WF_{\rm g}^s ( u ) = \emptyset$ if $u \in \cS (\rr d)$. 
Conversely, if $\WF_{\rm g}^s ( u ) = \emptyset$ then 
\begin{equation*}
\sup_{(x,\xi) \in \sr {2d-1}, \ \lambda > 0} \lambda^N |V_\fy u (\lambda x, \lambda^s \xi)| < + \infty \quad \forall N \geqs 0 
\end{equation*}
due to the compactness of the unit sphere $\sr {2d-1}$. 
Given $(y,\eta) \in T^* \rr d \setminus 0$
there is a unique $\lambda > 0$ such that $(y,\eta) = (\lambda x, \lambda^s \xi)$ and $(x,\xi) \in \sr {2d-1}$, 
and $|(y,\eta)|^2 = \lambda^2 |x|^2 + \lambda^{2s} |\xi|^2 \leqs \lambda^2 + \lambda^{2s}$. 
This implies that \eqref{eq:STFTschwartz} is satisfied, and thus $u \in \cS(\rr d)$. 
We have now shown that $\WF_{\rm g}^s ( u ) = \emptyset$ if and only if $u \in \cS(\rr d)$, for any $s > 0$.

\subsection{Window invariance and consequences}

First we show that $\WF_{\rm g}^s(u)$ does not depend on the window function $\fy \in \cS(\rr d) \setminus 0$. 

\begin{prop}\label{prop:windowinvariance}
Let $s > 0$, $u \in \cS'(\rr d)$ and $z_0 \in T^* \rr d \setminus 0$. 
If $\fy \in \cS(\rr d) \setminus 0$ and \eqref{eq:WFgs} holds with $t = 1$ for an open set $U \subseteq T^*\rr d \setminus 0$ containing $z_0$, 
and $\psi \in \cS(\rr d) \setminus 0$, then there exists an open set $V \subseteq U$ such that $z_0 \in V$ and 
\begin{equation}\label{eq:WFgs2}
\sup_{(x,\xi) \in V, \ \lambda > 0} \lambda^N |V_\psi u(\lambda x, \lambda^s \xi)| < \infty, \quad \forall N \geqs 0. 
\end{equation}
\end{prop}

\begin{proof}
Since $z_0 \in U \subseteq \rr {2d}$ where $U$ is open we may pick an open set $V \subseteq U$ 
such that $z_0 \in V$ and $V + \rB_\ep \subseteq U$ for some $0 < \ep \leqs 1$, and we may assume 
\begin{equation}\label{eq:Vbound}
\sup_{z \in V} |z| \leqs |z_0| + 1 := \mu. 
\end{equation}
By \cite[Lemma~11.3.3]{Grochenig1} we have 
\begin{equation*}
|V_\psi u (z)| \leqs (2 \pi)^{-\frac{d}{2}} \| \fy \|_{L^2}^{-2} \,  |V_\fy u| * |V_\psi \fy | (z), \quad z \in \rr {2d}.
\end{equation*}

Let $\lambda \geqs 1$ and $N \in \no$. We have
\begin{align*}
& \lambda^N |V_\psi u (\lambda x, \lambda^s \xi)| \\
& \lesssim \iint_{\rr {2d}} \lambda^N |V_\fy u ( \lambda (x- \lambda^{-1} y), \lambda^s (\xi - \lambda^{-s} \eta))| \ |V_\psi \fy (y,\eta) | \, \dd y \, \dd \eta \\
& = I_1 + I_2
\end{align*}
where we split the integral into the two terms 
\begin{align*}
I_1 = & \iint_{\rr {2d} \setminus \Omega_\lambda}  \lambda^N  |V_\fy u ( \lambda(x- \lambda^{-1} y), \lambda^s (\xi - \lambda^{-s} \eta))| \ |V_\psi \fy (y,\eta) | \, \dd y \, \dd \eta, \\
I_2 = & \iint_{\Omega_\lambda} \lambda^N |V_\fy u ( \lambda (x- \lambda^{-1} y), \lambda^s (\xi - \lambda^{-s} \eta))| \ |V_\psi \fy (y,\eta) | \, \dd y \, \dd \eta
\end{align*}
where  
\begin{equation*}
\Omega_\lambda = \{(y,\eta) \in \rr {2d}: |(y, \eta)| < 2^{-\frac12} \ep \lambda^{\min (1,s)}  \}. 
\end{equation*}

First we estimate $I_1$ when $(x,\xi) \in V$. 
From 
\eqref{eq:STFTtempered}, \eqref{eq:STFTschwartz} and \eqref{eq:Vbound}
we obtain for some $k \geqs 0$ and any $L \geqs k$
\begin{equation}\label{eq:estimateI1a}
\begin{aligned}
I_1 
& \lesssim \lambda^N \iint_{\rr {2d} \setminus \Omega_\lambda} 
\eabs{(\lambda x - y, \lambda^s \xi - \eta)}^k
  \, |V_\psi \fy (y,\eta) | \, \dd y \, \dd \eta \\
& \lesssim (1+ \mu^2 )^{\frac{k}{2}} \lambda^{N + k \max(1,s) } \iint_{\rr {2d} \setminus \Omega_\lambda} 
\eabs{(y, \eta)}^k  \, |V_\psi \fy (y,\eta) | \, \dd y \, \dd \eta \\
& \lesssim 
(1+ \mu^2 )^{\frac{k}{2}} \lambda^{N + k \max(1,s) }\iint_{\rr {2d} \setminus \Omega_\lambda} \eabs{(y, \eta)}^{k-L-2d-1}\, \dd y \, \dd \eta \\
& \lesssim 
\lambda^{N + k \max(1,s) }
\left( 1 + \frac12 \ep^2 \lambda^{2 \min(1,s)} \right)^{\frac12 \left( k -L\right)}
\iint_{\rr {2d}} \eabs{(y, \eta)}^{-2d-1}\, \dd y \, \dd \eta \\
& \lesssim 
\lambda^{N + k \max(1,s) + \min(1,s) ( k -L )} \\
& \leqs C_{N, L, \mu, \ep}
\end{aligned}
\end{equation}
for any $\lambda \geqs 1$, 
provided we pick $L \geqs k + \min(1,s)^{-1} \left( N + k\max(1,s) \right)$. 
Here $C_{N,L, \mu, \ep} > 0$ is a constant that depends on $N,L, \mu, \ep$ but not on $\lambda > 0$. 
Thus we have obtained the required estimate for $I_1$. 

It remains to estimate $I_2$. 
If $(y,\eta) \in \Omega_\lambda$ then $|y|^2 < \frac12 \ep^2 \lambda^2$ and $|\eta|^2 < \frac12 \ep^2 \lambda^{2s}$ which implies $(\lambda^{-1} y, \lambda^{-s} \eta) \in \rB_\ep$. 
Hence if $(x,\xi) \in V$ then $( x- \lambda^{-1} y, \xi - \lambda^{-s} \eta) \in U$ and we may use the estimate 
\eqref{eq:WFgs} with $t=1$. 
This gives 
\begin{equation}\label{eq:estimateI2a}
\begin{aligned}
I_2 & = \iint_{\Omega_\lambda} \lambda^N |V_\fy u ( \lambda (x- \lambda^{-1} y), \lambda^s (\xi - \lambda^{-s} \eta))| \, |V_\psi \fy (y,\eta) | \, \dd y \, \dd \eta \\
& \leqs C_{N} \iint_{\rr {2d}} | V_\psi \fy (y,\eta) | \, \dd y \, \dd \eta \\
& \lesssim C_{N}
\end{aligned}
\end{equation}
for all $\lambda \geqs 1$. 
Thus we have obtained the required estimate for $I_2$. 
Combining \eqref{eq:estimateI1a} and \eqref{eq:estimateI2a}, 
we have proved \eqref{eq:WFgs2}. 
\end{proof}

If $\check u(x) = u(-x)$ then 
\begin{equation}\label{eq:evensteven0}
V_{\check \psi} \check u(x,\xi)
= V_\psi u(-x,-\xi). 
\end{equation}
Using Proposition \ref{prop:windowinvariance}
it follows that we have the following symmetry:
\begin{equation}\label{eq:evensteven1}
\check u = \pm u \quad \Longrightarrow \quad \WF_{\rm g}^{s} (u) = - \WF_{\rm g}^{s} (u). 
\end{equation}

We also have 
\begin{equation}\label{eq:evensteven2}
V_{\psi} \overline{u}(x,\xi)
= \overline{V_{\overline \psi} u(x,-\xi)}. 
\end{equation}

Referring to \cite[Definition~3.2]{Rodino3} we observe that 
\begin{equation}\label{eq:WFgWFsinclusion}
\WF_{\rm g}^s (u) \subseteq \WF^{1,s} (u), \quad s > 0, \quad u \in \cS'(\rr d), 
\end{equation}
where $\WF^{1,s} (u)$ is a particular case of a $t,s$-Gelfand--Shilov wave front set, 
a concept that requires super-exponential rather than super-polynomial 
decay along curves in phase space.

\subsection{Metaplectic properties}

The Gabor wave front set is symplectically invariant as (cf. \cite[Proposition~2.2]{Hormander1})
\begin{equation}\label{eq:metaplecticWFg}
\WFg( \mu(\chi) u) = \chi \WFg(u), \quad \chi \in \Sp(d, \ro), \quad u \in \cS'(\rr d).
\end{equation}

When $s \neq 1$ the $s$-Gabor wave front set $\WFgs (u)$ is no longer symplectically invariant. 
Nevertheless, two of the generators of the symplectic group behave invariantly in certain individual senses which we now describe. 
By \cite[Proposition~4.10]{Folland1} each matrix $\chi \in \Sp(d,\ro)$ is a finite product of matrices in $\Sp(d,\ro)$ of the form
\begin{equation*}
\J, \quad 
\left(
  \begin{array}{cc}
  A^{-1} & 0 \\
  0 & A^{T}
  \end{array}
\right), 
\quad
\left(
  \begin{array}{cc}
  I & 0 \\
  B & I
  \end{array}
\right), 
\end{equation*}
for $A \in \GL(d,\ro)$ and $B \in \rr {d \times d}$ symmetric. 
The corresponding metaplectic operators are 
$\mu(\J) = \cF$, 
\begin{equation*}
\mu \left(
  \begin{array}{cc}
  A^{-1} & 0 \\
  0 & A^{T}
  \end{array}
\right) f(x)
= |A|^{\frac12} f(Ax), 
\end{equation*}
if $A \in \GL(d,\ro)$, and 
\begin{equation*}
\mu 
\left(
  \begin{array}{cc}
  I & 0 \\
  B & I
  \end{array}
\right) f(x)
= e^{\frac{i}{2} \la B x, x \ra} f(x), 
\end{equation*}
if $B \in \rr {d \times d}$ is symmetric.

\begin{prop}\label{prop:sGabormetaplectic}
Let $s > 0$ and $u \in \cS'(\rr d)$. 
Then we have 

\begin{enumerate}[\rm(i)]

\item 
\begin{equation*}
\WFgs (\wh u) = \J \WF_{\rm g}^{\frac1s} (u).  
\end{equation*}

\item If $A \in \GL(d,\ro)$ and $u_A (x) = |A|^{\frac12} u(Ax)$ then 
\begin{equation*}
\WFgs (u_A) = 
\left(
  \begin{array}{cc}
  A^{-1} & 0 \\
  0 & A^{T}
  \end{array}
\right) 
\WFgs (u). 
\end{equation*}

\item If $B \in \rr {d \times d}$ is symmetric and $v(x) = e^{\frac{i}{2} \la B x, x \ra} u(x)$ then if $s = 1$
\begin{equation}\label{eq:case3a}
\WFgs (v) = 
\left(
  \begin{array}{cc}
  I & 0 \\
  B & I
  \end{array}
\right) 
\WFgs (u), 
\end{equation}
if $s > 1$ then
\begin{equation}\label{eq:case3b}
\WFgs (v) = \WFgs (u), 
\end{equation}
and finally if $0 < s < 1$ then
\begin{equation}\label{eq:case3c}
(x,\xi) \in \WFgs (u) 
\quad \mbox{for some} \ \xi \in \rr d
\quad \Longrightarrow \quad (x,Bx) \in \WF_g (v). 
\end{equation}
\end{enumerate}

\end{prop}

\begin{proof}
Let $\fy \in \cS(\rr d) \setminus 0$.  
We have from the proof of \cite[Corollary~4.5]{Carypis1}
\begin{equation}\label{eq:STFTmetaplectic1}
|V_{\mu (\chi) \fy} (\mu(\chi) u)( \chi(x,\xi))| = |V_{\fy} u( x, \xi)|
\end{equation}
for all $\chi \in \Sp(d, \ro)$. 

\begin{enumerate}[\rm(i)]

\item
If $\chi = \J$ we obtain
\begin{equation*}
|V_{\wh \fy} \wh u(\J (x,\xi))| = |V_{\wh \fy} \wh u(\xi, - x)| = |V_{\fy} u( x,\xi) |. 
\end{equation*}
From this and Proposition \ref{prop:windowinvariance} it follows that $(x,\xi) \notin \WF_{\rm g}^{\frac1s} (u)$ if and only if $\J(x,\xi) \notin \WFgs (\wh u)$ which proves claim (i). 

\item
Next we insert $u_A$ for $A \in \GL(d,\ro)$ into \eqref{eq:STFTmetaplectic1} which gives
\begin{equation*}
|V_{\psi_A} u_A ( A^{-1}x, A^T \xi )| = |V_\psi u( x,\xi) |. 
\end{equation*}
Note that $\psi_A \in \cS(\rr d) \setminus 0$. 
We obtain $(x,\xi) \notin \WFgs (u)$ if and only if $(A^{-1}x, A^T \xi) \notin \WFgs (u_A)$ which shows claim (ii). 

\item
When $s=1$ \eqref{eq:case3a} is a particular case of \eqref{eq:metaplecticWFg}. 

Suppose $s \neq 1$. 
With $\psi(x) = e^{\frac{i}{2} \la B x, x \ra} \fy(x) \in \cS(\rr d) \setminus 0$
we obtain from \eqref{eq:STFTmetaplectic1}
\begin{equation*}
|V_{\fy} u( x, \xi)| = |V_{\psi} v(  x, B x + \xi) |
\end{equation*}
or equivalently 
\begin{equation*}
|V_{\fy} u( x, - B x + \xi)| = |V_{\psi} v(  x, \xi) |. 
\end{equation*}
If $\lambda > 0$ then 
\begin{equation}\label{eq:STFTmetaplectic2}
|V_{\fy} u( \lambda x, \lambda^s \xi)|
= |V_{\psi} v( \lambda x, \lambda^s ( \lambda^{1-s} B x + \xi))| 
= |V_{\psi} v( \lambda x, \lambda (B x + \lambda^{s-1} \xi ))|
\end{equation}
and 
\begin{equation}\label{eq:STFTmetaplectic3}
|V_{\psi} v( \lambda x, \lambda^s \xi)| 
= |V_{\fy} u( \lambda x, \lambda^s (- \lambda^{1-s} B x + \xi))| 
= |V_{\fy} u( \lambda x, \lambda (- B x + \lambda^{s-1} \xi))|. 
\end{equation}

Suppose $s > 1$ and $0 \neq (x_0, \xi_0) \notin \WFgs(u)$. 
Then for some $\ep > 0$ we have 
\begin{equation}\label{eq:WFgsnot1}
\sup_{x \in x_0 + \rB_\ep, \ \xi \in \xi_0 + \rB_{2\ep}, \ \lambda > 0} \lambda^N |V_{\fy} u( \lambda x, \lambda^s \xi)| < +\infty \quad \forall N \geqs 0. 
\end{equation}

We have $\lambda^{1-s} |B x| < \ep$ when $x \in x_0 + \rB_\ep$
if $\lambda \geqs L$ for $L \geqs 1$ sufficiently large.  
Thus $\xi - \lambda^{1-s} B x \in \xi_0 + \rB_{2 \ep}$ if $\xi \in \xi_0 + \rB_\ep$ and $\lambda \geqs L$. 
From \eqref{eq:STFTmetaplectic3} and \eqref{eq:WFgsnot1} we obtain
\begin{equation*}
\sup_{x \in x_0 + \rB_\ep, \ \xi \in \xi_0 + \rB_\ep, \ \lambda > 0} \lambda^N 
|V_{\psi} v( \lambda x, \lambda^s \xi)| < +\infty \quad \forall N \geqs 0
\end{equation*}
which shows that $(x_0, \xi_0) \notin \WFgs(v)$. 
Thus $\WFgs(v) \subseteq \WFgs(u)$. 
Likewise one shows the opposite inclusion using \eqref{eq:STFTmetaplectic2}. 
We have now proved \eqref{eq:case3b}. 

Suppose $0 < s < 1$ and $0 \neq (x_0, B x_0) \notin \WFg(v)$. 
Then for some $\ep > 0$ we have
\begin{equation}\label{eq:WFgsnot2}
\sup_{x \in x_0 + \rB_\ep, \ \xi \in B x_0 + \rB_{2 |B| \ep}, \ \lambda > 0} \lambda^N |V_{\psi} v( \lambda x, \lambda \xi)| < +\infty \quad \forall N \geqs 0. 
\end{equation}

Let $\eta_0 \in \rr d$.
We have $B x + \lambda^{s-1} \xi \in B x_0 + \rB_{2 |B| \ep}$ when $x \in x_0 + \rB_\ep$ and $\xi \in \eta_0 + \rB_\ep$
if $\lambda \geqs L$ for $L \geqs 1$ sufficiently large. 
From \eqref{eq:STFTmetaplectic2} we obtain 
\begin{equation*}
\sup_{x \in x_0 + \rB_\ep, \ \xi \in \eta_0 + \rB_\ep, \ \lambda > 0} \lambda^N 
|V_{\fy} u( \lambda x, \lambda^s \xi )| < +\infty \quad \forall N \geqs 0
\end{equation*}
and it follows that $(x_0, \eta_0) \notin \WFgs(u)$. 
We have shown \eqref{eq:case3c}. 
\end{enumerate}
\end{proof}

\section{Microlocality for anisotropic Gabor wave front sets}\label{sec:microlocalgabor}

Let $m \in \ro$, $0 \leqs \rho \leqs 1$, $a \in G_\rho^m$ and $\fy \in \cS(\rr {2d}) \setminus 0$. 
According to \cite[Proposition~3.2]{Cappiello3} the estimates 
\begin{equation}\label{eq:STFTvillkor1}
|V_\fy a(z,\zeta)| \lesssim \eabs{z}^m  \eabs{\zeta}^{-L}, \quad (z,\zeta) \in T^* \rr {2d}, 
\end{equation}
hold for any $L \geqs 0$. 
Note that the case $\rho = 0$ is included, so \eqref{eq:STFTvillkor1} is valid under the assumption
\begin{equation*}
\left| 
\pdd x \alpha \pdd \xi \beta a(x,\xi)
\right|
\lesssim \eabs{(x,\xi)}^m, \quad \alpha, \beta \in \nn d. 
\end{equation*}

The next result concerns microlocality with respect to the $s$-Gabor wave front set
for pseudodifferential operators in the isotropic Shubin calculus. 
Due to \eqref{eq:Gmsinclusion} the result is also true for the anisotropic Shubin symbols $G^{m,s}$. 

\begin{prop}\label{prop:microlocal}
Let $s > 0$, $m \in \ro$ and $0 \leqs \rho \leqs 1$. 
If $u \in \cS'(\rr d)$ and $a \in G_\rho^m$ then 
\begin{equation}\label{eq:microlocal1}
\WFgs ( a^w(x,D) u) \subseteq \WFgs ( u).  
\end{equation}
\end{prop}

\begin{proof}
Pick $\fy \in \cS(\rr d)$ such that $\| \fy \|_{L^2}=1$.
Denoting the formal adjoint of $a^w(x,D)$ by $a^w(x,D)^*$, \eqref{eq:moyal} gives for $u \in \cS' (\rr d)$ and $z \in \rr {2d}$
\begin{align*}
(2 \pi)^{\frac{d}{2}}  V_\varphi (a^w(x,D) u) (z)
& = ( a^w(x,D) u, \Pi(z) \varphi ) \\
& = ( u, a^w(x,D)^* \Pi(z) \varphi ) \\
& = \int_{\rr {2d}} V_\varphi u(w) \, ( \Pi(w) \varphi,a^w(x,D)^* \Pi(z) \varphi ) \, \dd w \\
& = \int_{\rr {2d}} V_\varphi u(w) \, ( a^w(x,D) \, \Pi(w) \varphi,\Pi(z) \varphi ) \, \dd w \\
& = \int_{\rr {2d}} V_\varphi u(z-w) \, ( a^w(x,D) \, \Pi(z-w) \varphi,\Pi(z) \varphi ) \, \dd w.
\end{align*}
By e.g. \cite[Lemma 3.1]{Grochenig2}, or a computation using \eqref{eq:wignerweyl}, we have 
\begin{equation*}
|( a^w(x,D) \, \Pi(z-w) \varphi,\Pi(z) \varphi )|
= \left| V_\Phi a \left( z-\frac{w}{2}, \J w \right) \right|
\end{equation*}
where $\Phi$ is the Wigner distribution $\Phi = W(\fy,\fy) \in \cS(\rr {2d})$. 

Combining the preceding identities we deduce
\begin{equation}\label{eq:STFTWeylop}
|V_\varphi (a^w(x,D) u) (z)|
\lesssim \int_{\rr {2d}}  |V_\varphi u(z-w)| \, \left| V_\Phi a \left( z-\frac{w}{2}, \J w \right) \right| \, \dd w.
\end{equation}

Suppose $0 \neq z_0 \notin \WF_{\rm g}^s(u)$. 
Then there exists an open set $U$ such that $z_0 \in U$ and 
\eqref{eq:WFgs} holds with $t = 1$. 
We pick an open set $V$ such that 
$z_0 \in V$ and $V + \rB_\ep \subseteq U$ for some $0 < \ep \leqs 1$, and we may assume 
that \eqref{eq:Vbound} holds. 

Let $\lambda \geqs 1$ and $N \in \no$. We have
\begin{align*}
& \lambda^N |V_\fy (a^w(x,D) u)  (\lambda x, \lambda^s \xi)| \\
& \lesssim \iint_{\rr {2d}} \lambda^N |V_\fy u ( \lambda (x- \lambda^{-1} y), \lambda^s (\xi - \lambda^{-s} \eta))| \, \left| V_\Phi a \left( \lambda x-\frac{y}{2}, \lambda^s \xi-\frac{\eta}{2},  \eta, -y \right) \right| \, \dd y \, \dd \eta \\
& = I_1 + I_2
\end{align*}
where the integral is decomposed into the two terms 
\begin{align*}
I_1 = & \iint_{\rr {2d} \setminus \Omega_\lambda} \lambda^N |V_\fy u ( \lambda (x- \lambda^{-1} y), \lambda^s (\xi - \lambda^{-s} \eta))| \, \left| V_\Phi a \left( \lambda x-\frac{y}{2}, \lambda^s \xi-\frac{\eta}{2},  \eta, -y \right) \right| \, \dd y \, \dd \eta, \\
I_2 = & \iint_{\Omega_\lambda} \lambda^N |V_\fy u ( \lambda (x- \lambda^{-1} y), \lambda^s (\xi - \lambda^{-s} \eta))| \, \left| V_\Phi a \left( \lambda x-\frac{y}{2}, \lambda^s \xi-\frac{\eta}{2},  \eta, -y \right) \right| \, \dd y \, \dd \eta
\end{align*}
where  
\begin{equation*}
\Omega_\lambda = \{(y,\eta) \in \rr {2d}: |(y, \eta)| < 2^{-\frac12} \ep \lambda^{\min (1,s)}  \}. 
\end{equation*}

First we estimate $I_1$ when $(x,\xi) \in V$. 
From 
\eqref{eq:STFTtempered}, \eqref{eq:Vbound} and \eqref{eq:STFTvillkor1} 
we obtain for some $k \geqs 0$ and any $L \geqs k + |m|$
\begin{equation}\label{eq:estimateI1b}
\begin{aligned}
I_1 
& \lesssim \lambda^N \iint_{\rr {2d} \setminus \Omega_\lambda} 
\eabs{(\lambda x - y, \lambda^s \xi - \eta)}^k
  \, \left| V_\Phi a \left( \lambda x-\frac{y}{2}, \lambda^s \xi-\frac{\eta}{2},  \eta, -y \right) \right| \, \dd y \, \dd \eta \\
& \lesssim (1+ \mu^2 )^{\frac{k}{2}} \lambda^{N + k \max(1,s) } \iint_{\rr {2d} \setminus \Omega_\lambda} 
\eabs{(y, \eta)}^k  \, \left| V_\Phi a \left( \lambda x-\frac{y}{2}, \lambda^s \xi-\frac{\eta}{2},  \eta, -y \right) \right| \, \dd y \, \dd \eta \\
& \lesssim 
(1+ \mu^2 )^{\frac{k+|m|}{2}} \lambda^{N + (k+|m|) \max(1,s) }\iint_{\rr {2d} \setminus \Omega_\lambda} \eabs{(y, \eta)}^{k+|m| -L - 2 d -1 }\, \dd y \, \dd \eta \\
& \lesssim 
\lambda^{N + (k + |m|) \max(1,s) }\iint_{\rr {2d} \setminus \Omega_\lambda} \eabs{(y, \eta)}^{k+|m| -L} \, \eabs{(y, \eta)}^{-2d-1}\, \dd y \, \dd \eta \\
& \leqs
\lambda^{N + (k + |m|) \max(1,s) }
\left( 1 + \frac12 \ep^2 \lambda^{2 \min(1,s)} \right)^{\frac12 \left( k + |m| -L\right)}
\iint_{\rr {2d}} \eabs{(y, \eta)}^{-2d-1}\, \dd y \, \dd \eta \\
& \lesssim 
\lambda^{N + (k+|m|) \max(1,s) + \min(1,s) ( k + |m| -L )} \\
& \leqs C_{N, L, a, \mu, \ep}
\end{aligned}
\end{equation}
for any $\lambda \geqs 1$, 
provided we pick $L \geqs k + |m|  + \min(1,s)^{-1} \left( N + (k+|m|) \max(1,s) \right)$. 
Here $C_{N,L,a, \mu, \ep} > 0$ is a constant that depends on $N,L,a, \mu, \ep$ but not on $\lambda > 0$. 
Thus we have obtained the required estimate for $I_1$. 

It remains to estimate $I_2$. 
If $(y,\eta) \in \Omega_\lambda$ then $|y|^2 < \frac12 \ep^2 \lambda^2$ and $|\eta|^2 < \frac12 \ep^2 \lambda^{2s}$ which implies $(\lambda^{-1} y, \lambda^{-s} \eta) \in \rB_\ep$. 
Hence if $(x,\xi) \in V$ then $( x- \lambda^{-1} y, \xi - \lambda^{-s} \eta) \in U$ and we may use the estimate 
\eqref{eq:WFgs} with $t=1$. 

This gives for any $L \geqs 0$ and a constant $C_{N,s,m} > 0$, using 
\eqref{eq:STFTvillkor1} and \eqref{eq:Vbound}
\begin{equation}\label{eq:estimateI2b}
\begin{aligned}
I_2 & = \iint_{\Omega_\lambda} \lambda^N |V_\fy u ( \lambda (x- \lambda^{-1} y), \lambda^s (\xi - \lambda^{-s} \eta))| \, \left| V_\Phi a \left( \lambda x-\frac{y}{2}, \lambda^s \xi-\frac{\eta}{2},  \eta, -y \right) \right| \, \dd y \, \dd \eta \\
& = \lambda^{- |m| \max(1,s)} \iint_{\Omega_\lambda} \lambda^{N + |m| \max(1,s)} |V_\fy u ( \lambda (x- \lambda^{-1} y), \lambda^s (\xi - \lambda^{-s} \eta))| \\
& \qquad \qquad \qquad \qquad \qquad \qquad \qquad \qquad \qquad 
\times \left| V_\Phi a \left( \lambda x-\frac{y}{2}, \lambda^s \xi-\frac{\eta}{2},  \eta, -y \right) \right| \, \dd y \, \dd \eta \\
& \leqs C_{N,s,m} \lambda^{- |m| \max(1,s)} \iint_{\Omega_\lambda} \left| V_\Phi a \left( \lambda x-\frac{y}{2}, \lambda^s \xi-\frac{\eta}{2},  \eta, -y \right) \right| \, \dd y \, \dd \eta \\
& \leqs C_{N,s, m} \lambda^{- |m| \max(1,s) + |m| \max(1,s)} \iint_{\rr {2d}} \eabs{(y,\eta)}^{|m|-L} \, \dd y \, \dd \eta \\
& \lesssim C_{N,s,m}
\end{aligned}
\end{equation}
provided $L > |m| + 2 d$, 
for all $\lambda \geqs 1$. 
Thus we have obtained the required estimate for $I_2$. 
Combining \eqref{eq:estimateI1b} and \eqref{eq:estimateI2b}, 
referring to Definition \ref{def:WFgs}, 
we may conclude that $z_0 \notin \WF_{\rm g}^{s}( a^w(x,D) u )$
and hence we have proved \eqref{eq:microlocal1}. 
\end{proof}

A consequence of Proposition \ref{prop:microlocal} is the invariance of the anisotropic Gabor wave front set under translations and modulations, 
a k a time-frequency shifts \cite{Grochenig1}. 

\begin{cor}\label{cor:translationmodulationinvar}
Suppose $s > 0$. 
For any $z \in \rr {2d}$ and any $u \in \cS'(\rr d)$ we have
\begin{equation*}
\WFgs ( \Pi(z) u ) = \WFgs (u). 
\end{equation*}
\end{cor}

\begin{proof}
Let $z = (x,\xi) \in \rr {2d}$. 
By a calculation it is verified that $\Pi(x,\xi) = a_{x,\xi}^w(x,D)$ where
\begin{equation*}
a_{x,\xi} (y,\eta) = e^{ \frac{i}{2} \la x, \xi \ra + i \left( \la y, \xi \ra - \la x, \eta \ra \right)}, \quad (y,\eta) \in \rr {2d}. 
\end{equation*}
For any $\alpha, \beta \in \nn d$ we have
\begin{equation*}
\left| \partial_y^\alpha \partial_\eta^\beta a_{x,\xi} (y,\eta) \right|
= |\xi^\alpha x^\beta| := C_{\alpha, \beta}
\end{equation*}
where we may consider $|\xi^\alpha x^\beta| \geqs 0$ as a constant as a function of $(y,\eta) \in \rr {2d}$.
This implies that $a_{x,\xi} \in G_0^0$. 
Thus we may apply Proposition \ref{prop:microlocal} which gives 
\begin{equation*}
\WFgs ( \Pi(z) u ) \subseteq \WFgs (u). 
\end{equation*}
The opposite inclusion follows from $u = e^{- i \la x, \xi \ra} \Pi(-(x,\xi)) \Pi(x,\xi) u$. 
\end{proof}

We finish this section with the anisotropic Gabor wave front sets for a few important tempered distributions. 

\begin{prop}\label{prop:WFstelementary}
If $s > 0$ then:

\begin{enumerate}[\rm (i)]

\item for any $x \in \rr d$ and any $\alpha \in \nn d$
\begin{equation}\label{eq:diracWFs}
\WFgs ( D^\alpha \delta_x ) =  \{ 0 \} \times ( \rr d \setminus 0 ); 
\end{equation}
\item for any $\alpha \in \nn d$
\begin{equation*}
\WFgs ( x^\alpha ) =  ( \rr d \setminus 0 ) \times \{ 0 \}; 
\end{equation*}
\item for any $\xi \in \rr d$
\begin{equation*}
\WFgs (e^{i \la \cdot, \xi \ra} ) =  ( \rr d \setminus 0 ) \times \{ 0 \}. 
\end{equation*}
\end{enumerate}

\end{prop}

\begin{proof}
Due to Corollary \ref{cor:translationmodulationinvar} we may assume $x=0$ in (i) and $\xi=0$ in (iii). 
By Proposition \ref{prop:sGabormetaplectic} (i) it suffices to show (i), since $\wh {D^ \alpha \delta_0} (\xi) = (2 \pi)^{-\frac{d}{2}} \xi^{\alpha}$. 

Let $\fy \in \cS (\rr d)$ satisfy $\fy \equiv 1$ in a neighborhood of the origin. 
We have
\begin{equation*}
V_\fy D^ \alpha \delta_0 (x,\xi) = (2 \pi)^{- \frac{d}{2}} \sum_{\beta \leqs \alpha} \binom{\alpha}{\beta} \xi^\beta \overline{ D^{\alpha-\beta} \fy (-x)}. 
\end{equation*}

If $\xi \neq 0$ we obtain for $\lambda > 0$
\begin{equation*}
V_\fy D^ \alpha \delta_0 (0, \lambda^s \xi) = (2 \pi)^{- \frac{d}{2}} \lambda^{s |\alpha|} \xi^\alpha
\end{equation*}
which does not decay as a function of $\lambda$. Thus 
\begin{equation}\label{eq:diracWFs1}
\{ 0 \} \times ( \rr d \setminus 0 ) \subseteq \WFgs ( D^\alpha \delta_0 ). 
\end{equation}

Suppose on the other hand $(x_0, \xi_0) \in T^* \rr d$ and $x_0 \neq 0$. 
If $0 < \ep < |x_0|/2$,  $x \in x_0 + \rB_\ep$, $\xi \in \xi_0 + \rB_\ep$ and $\lambda \geqs 1$ then 
for any $n \in \no$ we have 
\begin{align*}
|V_\fy D^ \alpha \delta_0 (\lambda x, \lambda^s \xi) |
& = (2 \pi)^{- \frac{d}{2}} \left| \sum_{\beta \leqs \alpha} \binom{\alpha}{\beta} \lambda^{s |\beta|} \xi^\beta \overline{ D^{\alpha-\beta} \fy (- \lambda x)} \right| \\
& \lesssim \sum_{\beta \leqs \alpha} \binom{\alpha}{\beta} \lambda^{s |\beta|} |\xi|^\beta \eabs{ \lambda x}^{-n} \\
& \lesssim \lambda^{s |\alpha| - n}. 
\end{align*}
This shows
\begin{equation*}
\WFgs ( D^\alpha \delta_0 )
\subseteq 
\{ 0 \} \times ( \rr d \setminus 0 ) 
\end{equation*}
so combining with \eqref{eq:diracWFs1} we have shown \eqref{eq:diracWFs} when $x = 0$. 
\end{proof}

\section{Microellipticity for anisotropic Gabor wave front sets}\label{sec:microellipticgabor}

The main result in this section is the microelliptic inclusion expressed in Theorem \ref{thm:microelliptic}. 
To get there we need a definition and several auxiliary results. 

\begin{defn}\label{def:csupp}
Suppose $s > 0$, $a \in G^{m,s}$ and let $p$ be the projection \eqref{eq:projection}. 
The $s$-conical support $\csupp_{s} (a) \subseteq T^* \rr d \setminus 0$ of $a$ is
defined as follows. 
A point $z_0 \in T^* \rr d \setminus 0$ satisfies $z_0 \notin \csupp_{s} (a)$
if there exists $\ep > 0$ such that 
\begin{align*}
& \, \supp (a) \cap \overline{ \{ z \in \rr {2d} \setminus 0, \ | p(z) - p(z_0) | < \ep \} } \\
= & \, \supp (a) \cap \overline{ \Gamma }_{p(z_0), \ep} 
\quad \mbox{is compact in} \quad \rr {2d}. 
\end{align*}
\end{defn}

Clearly $\csupp_{s} (a) \subseteq T^* \rr d \setminus 0$ is $s$-conic. 

\begin{prop}\label{prop:sconesupp}
Let $s > 0$. 
If $u \in \cS'(\rr d)$ and $a \in G_0^m$ then 
\begin{equation}\label{eq:WFcsuppinclusion}
\WFgs ( a^w(x,D) u) \subseteq \csupp_{s} (a). 
\end{equation}
\end{prop}

\begin{proof}
We have $|\pd \beta a(w)| \lesssim \eabs{w}^{m}$ 
for any $\beta \in \nn {2d}$.
We may assume that $\csupp_{s} (a) \neq T^* \rr d \setminus 0$ since the inclusion is trivial otherwise. 
Let $0 \neq z_0 \notin \csupp_{s} (a)$. 
We may assume $|z_0| = 1$. 

By Lemma \ref{lem:sconesequiv} we may assume that 
\begin{equation}\label{eq:suppa}
\supp (a) \subseteq \rB_R \cup \, \left( \rr {2d} \setminus \wt \Gamma_{z_0, 2 \ep} \right)
\end{equation}
for some $R > 0$ and $0 < \ep < 1$. 

Let $\fy \in \cS(\rr d) \setminus 0$ and set $\Phi = W(\fy,\fy) \in \cS(\rr {2d})$. 
We start by proving the following estimate for any $\alpha, \beta \in \nn {2d}$ such that $\beta \leqs \alpha$, any $\lambda \geqs 1$, $(x,\xi) \in z_0 + \rB_\ep$, and any $L \geqs | m | + 2d + 1$. 
We have
\begin{equation}\label{eq:STFTconesupp1}
\int_{\rr {2d}} \left| \pd \beta a(w) \right| \, \left| \partial^{\alpha-\beta} \Phi \left( w - \left( \lambda x - \frac{y}{2}, \lambda^s \xi - \frac{\eta}{2} \right) \right) \right| \, \dd w
\lesssim \lambda^{-L \min(1,s) + | m | \max(1,s) } \eabs{(y,\eta)}^{2L}. 
\end{equation}

In fact using Peetre's inequality we obtain on the one hand for any $L \geqs 0$
\begin{equation}\label{eq:STFTconesupp2}
\begin{aligned}
& \int_{\rB_R} \left| \pd \beta a(w) \right| \, \left| \partial^{\alpha-\beta} \Phi \left( w - \left( \lambda x - \frac{y}{2}, \lambda^s \xi - \frac{\eta}{2} \right)  \right) \right| \, \dd w \\
& \lesssim \int_{\rB_R} \left\langle w - \left( \lambda x - \frac{y}{2}, \lambda^s \xi - \frac{\eta}{2} \right)  \right\rangle^{-L} \, \dd w
\lesssim \int_{\rB_R} \eabs{w}^L \eabs{ (y,\eta) }^L \eabs{ (\lambda x, \lambda^s \xi ) }^{-L} \, \dd w \\
& \lesssim \lambda^{-L \min(1,s)} \eabs{(y,\eta)}^L. 
\end{aligned}
\end{equation}

On the other hand, since 
\begin{equation*}
\left| \left( \lambda^{-1} u, \lambda^{-s} \theta \right) - z_0  \right| \geqs 2 \ep \quad \forall \lambda > 0 \quad \forall (u,\theta) \in \rr {2d} \setminus \wt \Gamma_{z_0, 2 \ep} 
\end{equation*}
we have for $(x,\xi) \in z_0 + \rB_\ep$
\begin{equation*}
\left| \left( \lambda^{-1} u, \lambda^{-s} \theta \right) - (x,\xi) \right| \geqs \ep \quad \forall \lambda > 0 \quad \forall (u,\theta) \in \rr {2d} \setminus \wt \Gamma_{z_0, 2 \ep}. 
\end{equation*}
It follows that for $\lambda \geqs 1$, $(x,\xi) \in z_0 + \rB_\ep$ and 
$w = (u,\theta) \in \rr {2d} \setminus \wt \Gamma_{z_0, 2 \ep}$ we have
\begin{align*}
\left| w - ( \lambda x, \lambda^s \xi) \right|^2 
& = \lambda^2 | \lambda^{-1} u - x |^2 + \lambda^{2s} | \lambda^{-s} \theta - \xi |^2 \\
& \geqs \lambda^{2 \min (1,s)} \ep^2. 
\end{align*}

This gives for $(x,\xi) \in z_0 + \rB_\ep$ and any $L \geqs  | m | + 2d + 1$
\begin{equation}\label{eq:STFTconesupp3}
\begin{aligned}
& \int_{\rr {2d} \setminus \wt \Gamma_{z_0, 2 \ep} } \left| \pd \beta a(w) \right| \, \left| \partial^{\alpha-\beta} \Phi \left( w - \left( \lambda x - \frac{y}{2}, \lambda^s \xi - \frac{\eta}{2} \right) \right) \right| \, \dd w \\
& \lesssim \int_{\rr {2d} \setminus \wt \Gamma_{z_0, 2 \ep} } \eabs{w}^{|m|} 
\left\langle w - \left( \lambda x - \frac{y}{2}, \lambda^s \xi - \frac{\eta}{2} \right) \right\rangle^{- 2 L} \, \dd w \\
& \lesssim \eabs{(y,\eta)}^{2 L} \int_{\rr {2d} \setminus \wt \Gamma_{z_0, 2 \ep} } \eabs{w}^{|m|} 
\left\langle w - \left( \lambda x, \lambda^s \xi \right)  \right\rangle^{- L} 
\left\langle w - \left( \lambda x, \lambda^s \xi \right)  \right\rangle^{- L} 
\, \dd w \\
& \lesssim \lambda^{-L \min(1,s)} \eabs{(y,\eta)}^{2 L}
\int_{\rr {2d}} 
\left\langle w + \left( \lambda x, \lambda^s \xi \right) \right\rangle^{|m|} 
\eabs{w}^{-L} \, \dd w \\
& \lesssim \lambda^{-L \min(1,s)} \eabs{(y,\eta)}^{2 L}
\left\langle \left( \lambda x, \lambda^s \xi \right) \right\rangle^{|m|} 
\int_{\rr {2d}} 
\eabs{w}^{|m|-L} \, \dd w \\
&\lesssim \lambda^{-L \min(1,s) + | m | \max(1,s) } \eabs{(y,\eta)}^{2L}. 
\end{aligned}
\end{equation}
Combining \eqref{eq:suppa}, \eqref{eq:STFTconesupp2} and \eqref{eq:STFTconesupp3} we have now shown 
\eqref{eq:STFTconesupp1}.

Next we observe that integration by parts gives for any $\alpha \in \nn {2d}$ and $z, \zeta \in \rr {2d}$
\begin{align*}
\left| \zeta^\alpha V_\Phi a \left( z, \zeta \right) \right|
& = (2 \pi)^{-d} \left| \int_{\rr {2d}} a(w) \pdd w \alpha \left( e^{- i \la w,\zeta \ra} \right) \overline{\Phi(w-z)} \dd w \right| \\
& \lesssim \sum_{\beta \leqs \alpha} \binom{\alpha}{\beta} \int_{\rr {2d}} \left| \pd \beta a(w) \right| \, \left| \partial^{\alpha-\beta} \Phi(w - z) \right| \, \dd w. 
\end{align*}

Combining this with \eqref{eq:STFTconesupp1} we obtain for $(x,\xi) \in z_0 + \rB_\ep$, and any $M \in \no$, 
$L \geqs | m | + 2d + 1$ and $\lambda \geqs 1$
\begin{align*}
& \eabs{(y,\eta)}^{2(M+L)} \left| V_\Phi a \left( \lambda x-\frac{y}{2}, \lambda^s \xi-\frac{\eta}{2},  \eta, -y \right) \right| \\
& \lesssim \max_{|\alpha| \leqs 2(M+L)} \left| (y,\eta)^{\alpha} V_\Phi a \left( \lambda x-\frac{y}{2}, \lambda^s \xi-\frac{\eta}{2},  \eta, -y \right) \right| \\
& \lesssim \max_{|\alpha| \leqs 2(M+L)}
\sum_{\beta \leqs \alpha} \binom{\alpha}{\beta} \int_{\rr {2d}} \left| \pd \beta a(w) \right| \, \left| \partial^{\alpha-\beta} \Phi \left( w - \left( \lambda x-\frac{y}{2}, \lambda^s \xi-\frac{\eta}{2} \right) \right) \right| \, \dd w \\
& \lesssim \lambda^{-L \min(1,s) + | m | \max(1,s) } \eabs{(y,\eta)}^{2L}. 
\end{align*}

Given any $N, M \geqs 0$ we may pick $L \geqs 0$ such that 
$L \min(1,s) - | m | \max(1,s) \geqs N$. 
We thus have 
\begin{equation}\label{eq:STFTconesupp4}
\left| V_\Phi a \left( \lambda x-\frac{y}{2}, \lambda^s \xi-\frac{\eta}{2},  \eta, -y \right) \right|
\lesssim \lambda^{-N} \eabs{(y,\eta)}^{-M}
\end{equation}
for any $N, M \geqs 0$, $\lambda \geqs 1$ and $(x,\xi) \in z_0 + \rB_\ep$. 

Finally we prove that $z_0 \notin \WFgs ( a^w(x,D) u)$. 
We use \eqref{eq:STFTWeylop} from the proof of Proposition \ref{prop:microlocal} and \eqref{eq:STFTconesupp4}. 
This gives for $(x,\xi) \in z_0 + \rB_\ep$, using \eqref{eq:STFTtempered} for some $k \geqs 0$, for any $N, M \geqs 0$, $\lambda \geqs 1$
\begin{align*}
& |V_\varphi (a^w(x,D) u) ( \lambda x, \lambda^s \xi )| \\
& \lesssim 
\int_{\rr {2d}}  |V_\fy u \left( (\lambda x, \lambda^s \xi) - (y,\eta) \right)| \, \left| V_\Phi a \left( \lambda x - \frac{y}{2}, \lambda^s \xi - \frac{\eta}{2}, \eta, - y \right) \right| \, \dd y \, \dd \eta \\
& \lesssim \eabs{ (\lambda x, \lambda^s \xi) }^k
\int_{\rr {2d}}  \eabs{(y,\eta) }^k \, \left| V_\Phi a \left( \lambda x - \frac{y}{2}, \lambda^s \xi - \frac{\eta}{2}, \eta, - y \right) \right| \, \dd y \, \dd \eta \\
& \lesssim \lambda^{k \max (1,s) - N}
\end{align*}
provided $M \geqs k + 2 d + 1$. 
Since $N \geqs 0$ is arbitrary we have shown $z_0 \notin \WFgs ( a^w(x,D) u)$, 
and thus \eqref{eq:WFcsuppinclusion}. 
\end{proof}

As another tool for the microellipticity result Theorem \ref{thm:microelliptic} we need the following lemma
where we use Definition \ref{def:noncharacteristic}. 

\begin{lem}\label{lem:microelliptic}
Suppose $s > 0$, $a \in G^{m,s}$ and $\charac_{s,m_1} (a) \neq T^* \rr d \setminus 0$ for some $m_1 \leqs m$. 
Let $\Gamma \subseteq T^* \rr d \setminus 0$ be a closed $s$-conic set such that 
$\charac_{s,m_1} (a) \cap \Gamma=\emptyset$. 
Then there exists $\rho > 0$ such that for any 
$\chi \in G^{0,s}$ with $\supp(\chi) \subseteq \Gamma \setminus \rB_\rho$,
there exists $b \in G^{-m_1,s}$ such that 
\begin{equation*}
b \wpr a = \chi + r 
\end{equation*}
where $r \in \cS(\rr {2d})$. 
\end{lem}
\begin{proof}
The proof follows established principles in pseudodifferential calculus. 
Therefore we content ourselves with a sketch of the main steps of the construction of the microlocal parametrix $b$.

As a first approximation set $b_0:= a^{-1} \chi$. The estimates 
\begin{equation*}
|\pdd x \alpha \pdd \xi \beta (a^{-1}) (x,\xi)| 
\leqs C_{\alpha \beta} |a(x,\xi)|^{-1} \mu_s (x,\xi)^{-|\alpha| - s |\beta|}, \quad \alpha, \beta \in \nn d, \quad (x,\xi) \in \Gamma, \quad |x| + |\xi|^{\frac1s}\geqs R, 
\end{equation*}
are consequences of the non-characteristic estimates 
\eqref{eq:lowerbound1}, \eqref{eq:boundderivative1} and induction.   

By Leibniz' rule they imply the estimates
\begin{equation*}
|\pdd x \alpha \pdd \xi \beta b_0 (x,\xi)| 
\leqs C_{\alpha \beta} |a(x,\xi)|^{-1} \mu_s (x,\xi)^{-|\alpha| - s |\beta|}, \quad \alpha, \beta \in \nn d, \quad |x| + |\xi|^{\frac1s} \geqs R, 
\end{equation*}
and consequently $b_0 \in G^{-m_1,s}$ if $\rho > 0$ is sufficiently large. 

Then, by \eqref{eq:calculuscomposition1} and again the non-characteristic estimates 
\eqref{eq:lowerbound1} and \eqref{eq:boundderivative1} it follows that $b_0\wpr a=\chi + r_0 + r_{0,\cS}$ with $r_0\in G^{-(1+s),s}$ satisfying $\supp(r_0) \subseteq \supp(\chi)$ and $r_{0,\cS}\in\cS(\rr {2d})$. Subsequently, setting $b_1 := - a^{-1} r_0$, we notice that 
we obtain the estimates
\begin{equation*}
|\pdd x \alpha \pdd \xi \beta b_1 (x,\xi)| \leqs C_{\alpha \beta} |a(x,\xi)|^{-1} \mu_s (x,\xi)^{-(1+s)-|\alpha|- s |\beta|}, \quad \alpha,\beta \in \nn d, \quad |x| + |\xi|^{\frac1s} \geqs R, 
\end{equation*}
and consequently $b_1\in G^{- m_1 - (1+s),s}$. 

This gives
\begin{equation*}
( b_0 + b_1 ) \wpr a = \chi + r_0 + r_{0,\cS} -r_0 + r_1 + r_{1,\cS} = \chi + r_1 + r_{0,\cS} + r_{1,\cS}
\end{equation*}
with $r_1\in G^{- 2(1+s),s}$, $\supp(r_1) \subseteq \supp(\chi)$ and $r_{1,\cS}\in\cS(\rr {2d})$. Constructing in this way recursively  
$b_{j+1}:=-a^{-1} r_j \in G^{-m_1- (s+1)(j+1),s}$ and $r_{j+1} \in G^{-(s+1)(j+2),s}$ with $\supp(r_{j+1}) \subseteq \supp(\chi)$, $j=1,2,\dots$, one obtains a sequence of symbols $(b_j)_{j \geqs 0}$.

Finally set $b\sim \sum_{j=0}^\infty b_j \in G^{-m_1,s}$. The symbol $b$ satisfies $b \wpr a = \chi + r$ with $r \in \cS(\rr {2d})$.
\end{proof}

Finally we are in a position to state and prove the main result on microellipticity in the anisotropic Shubin calculus. 
The proof is short due to the long preparation. 
Note that we require that the symbol is anisotropic, as opposed to Proposition \ref{prop:microlocal} where 
the symbol is allowed to be isotropic.  

\begin{thm}\label{thm:microelliptic}
Let $s > 0$. 
If $u \in \cS'(\rr d)$ and $a \in G^{m,s}$ then for any $m_1 \leqs m$
\begin{equation*}
\WFgs (u) \subseteq \WFgs ( a^w(x,D) u ) \bigcup \charac_{s,m_1}(a). 
\end{equation*}
\end{thm}

\begin{proof}
We may assume that $\WFgs ( a^w(x,D) u ) \neq T^* \rr d \setminus 0$ and $\charac_{s,m_1}(a) \neq T^* \rr d \setminus 0$, 
since the inclusion is trivial otherwise.
Let $0 \neq z_0 \notin \WFgs ( a^w(x,D) u )$
and $z_0 \notin \charac_{s,m_1}(a)$. 
Due to $s$-conic invariance we may assume $| z_0 | = 1$. 

Pick $\ep > 0$ such that 
$\overline{\Gamma}_{z_0,2 \ep} \cap \charac_{s,m_1}(a) = \emptyset$,
and pick $\chi \in G^{0,s}$ such that $\supp \chi \subseteq \Gamma_{z_0,2 \ep} \setminus \rB_R$
and $\chi |_{\Gamma_{z_0,\ep} \setminus \overline{\rB}_{2R} } \equiv 1$, 
for $R > 0$ to be chosen.  
This is possible due to 
Lemma \ref{lem:cutoff}. 
Then $z_0 \notin \csupp_s (1-\chi)$, and
due to \eqref{eq:Gmsinclusion} we have $\chi \in G_0^{0}$.
By Proposition \ref{prop:sconesupp} we may thus conclude 
\begin{equation*}
z_0 \notin \WFgs( (1-\chi)^w(x,D) u). 
\end{equation*}

According to Lemma \ref{lem:microelliptic} 
we may pick $R > 0$ such that 
there exists $b \in G^{-m_1,s}$ and $r \in \cS(\rr {2d})$ such that 
$1 = b \wpr a + r + 1 -\chi$, so we have 
\begin{equation*}
u = b^w(c,D) a^w(x,D) u + r^w(x,D) u + (1-\chi)^w(x,D) u. 
\end{equation*}
Here $r^w(x,D) u \in \cS(\rr d)$ which means that $z_0 \notin \WFgs ( r^w(x,D) u)$ trivially. 
By Proposition \ref{prop:microlocal} we have $z_0 \notin \WFgs ( b^w(x,D) a^w(x,D) u )$. 
Thus we may conclude that $z_0 \notin \WFgs (u)$. 
\end{proof}

\begin{cor}\label{cor:WFspropagation}
Let $s > 0$. 
If $u \in \cS'(\rr d)$, $a \in G^{m,s}$ and $\charac_{s,m_1}(a) = \emptyset$ for some $m_1 \leqs m$
then 
\begin{equation*}
\WFgs ( a^w(x,D) u ) = \WFgs (u).  
\end{equation*}
\end{cor}

\section{The $s$-Gabor wave front set of oscillatory functions}\label{sec:chirp}

An important reason for the introduction of the anisotropic Gabor wave front set $\WFgs (u)$ is that it describes accurately the phase space singularities of oscillatory functions known generically as chirp signals. 

Let $\fy: \rr d \to \ro$ be a real polynomial of order $m \geqs 2$
\begin{equation}\label{eq:phasefunction}
\fy (x) = \fy_m(x) + p(x)
\end{equation}
where 
\begin{equation}\label{eq:polynomial1}
p (x) = \sum_{0 \leqs |\alpha| < m} c_\alpha x^\alpha, \quad c_\alpha \in \ro, 
\end{equation}
and 
\begin{equation}\label{eq:principalpart}
\fy_m(x) = \sum_{|\alpha| = m} c_\alpha x^\alpha, \quad c_\alpha \in \ro, \quad \exists \alpha \in \nn d: \  |\alpha| = m, \ c_\alpha \in \ro \setminus 0, 
\end{equation}
is the principal part. 

We will study chirp functions of the form 
\begin{equation}\label{eq:chirpdef}
u(x) = e^{i \fy(x)}, \quad x \in \rr d. 
\end{equation}

First we note that for any $\lambda > 0$ and any $1 \leqs j \leqs d$ we have 
\begin{equation}\label{eq:phasederivative1}
\lambda^{-m} \partial_j \left( \fy(\lambda y) \right)
= \partial_j \fy_m (y) + \lambda^{1-m} \partial_j p(\lambda y) 
\end{equation}
and if $|y| \leqs R$ and $\lambda \geqs 1$ then 
\begin{equation}\label{eq:polderivative1}
\lambda^{1-m} | \partial_j p(\lambda y) |
= \left| \sum_{0 \leqs |\alpha| \leqs m-1} \alpha_j c_\alpha y^{\alpha-e_j} \lambda^{|\alpha| - m} \right|
\leqs C_R \lambda^{-1}. 
\end{equation}

The following result shows that only the principal part $\fy_m(x)$ of $\fy$ 
is recorded in $\WF_{\rm g}^{m-1} (u)$, and the $(m-1)$-Gabor wave front set is contained in the $(m-1)$-conic set in phase space which is  
the graph of its gradient, that is $0 \neq x \mapsto ( x , \nabla \fy_m (x))$. 
The gradient of the phase function is known as the instantaneous frequency \cite{Boggiatto1}. 

\begin{thm}\label{thm:chirpWFgs}
If $m \geqs 2$ and $\fy$ is a real polynomial defined by \eqref{eq:phasefunction}, \eqref{eq:polynomial1} and \eqref{eq:principalpart},
and $u$ is defined by \eqref{eq:chirpdef}, then 
\begin{equation}\label{eq:chirpconclusion1}
\WF_{\rm g}^{m-1} (u) 
\subseteq \{ (x, \nabla \fy_m (x) ) \in \rr {2d}: \ x \neq 0 \}. 
\end{equation}
If $d = 1$ and $\fy$ is even or odd then 
\begin{equation}\label{eq:chirpconclusion2}
\WF_{\rm g}^{m-1} (u) 
= \{ (x, \fy_m'(x) ) \in \rr 2: \ x \neq 0 \}. 
\end{equation}
\end{thm}

\begin{proof}
Set 
\begin{align*}
W 
& = \{ (x, \nabla \fy_m (x) ) \in \rr {2d}: \ x \in \rr d \setminus 0 \} \subseteq T^* \rr d \setminus 0.
\end{align*}
Then $W$ is an $(m-1)$-conic set in $T^* \rr d \setminus 0$.

Suppose $(x_0, \xi_0) \in \rr {2d} \setminus 0$ and $(x_0, \xi_0) \notin W$. 
Then there exists $1 \leqs j \leqs d$ such that $\xi_{0,j} \neq \partial_j \fy_m (x_0)$. 
Thus there exists an open set $U$ such that $(x_0,\xi_0) \in U$, and $0 < \ep \leqs 1$, $\delta > 0$,
such that
\begin{equation*}
(x,\xi) \in U,  \quad |x-y| \leqs \delta \sqrt{2} \quad \Longrightarrow \quad  |\xi_j - \partial_j \fy_m (x) | \geqs 2 \ep, \quad | \partial_j (\fy_m (x) - \fy_m(y) )| \leqs \frac{\ep}{2}.
\end{equation*}
By \eqref{eq:polderivative1} we have 
\begin{equation*}
\lambda^{1-m} | \partial_j p (\lambda y) |
\leqs \frac{\ep}{2}
\end{equation*}
if $(x,\xi) \in U$, $|x-y| \leqs \delta \sqrt{2}$ and $\lambda \geqs L$ where $L \geqs 1$ is sufficiently large. 

Using \eqref{eq:phasederivative1} we obtain if $(x,\xi) \in U$, $|x-y| \leqs \delta \sqrt{2}$ and $\lambda \geqs L$ 
\begin{equation}\label{eq:lowerboundfas}
\left| \xi_j -  \lambda^{-m} \partial_j \left( \fy( \lambda y) \right)  \right| 
\geqs |\xi_j - \partial_j \fy_m (x) | - \left( | \partial_j ( \fy_m (y) - \fy_m (x) ) | + \lambda^{1-m} | \partial_j p (\lambda y) | \right)
\geqs \ep.
\end{equation}

Let $\psi \in C_c^\infty(\rr d) \setminus 0$ have $\supp \psi \subseteq \rB_\delta$. 
We denote by $y' \in \rr {d-1}$ the vector $y \in \rr d$ except coordinate $j$. 
The stationary phase theorem \cite[Theorem~7.7.1]{Hormander0}
gives, for any $k \in \no$, and any $\lambda \geqs L$, if $(x,\xi) \in U$, using \eqref{eq:lowerboundfas},
\begin{equation*}
\begin{aligned}
|V_\psi u ( \lambda x, \lambda^{m-1} \xi)|
& = (2 \pi)^{-\frac{d}{2}} \left| \int_{\rr d} e^{i ( \fy(y)  - \lambda^{m-1} \la y,  \xi \ra )} \overline{\psi( \lambda (\lambda^{-1} y-x) )} \, \dd y \right| \\
& = (2 \pi)^{-\frac{d}{2}} \lambda^d \left| \int_{| x - y | \leqs \delta} e^{i \lambda^{m} ( \lambda^{-m} \fy(\lambda y) - \la y, \xi \ra )} \overline{\psi( \lambda (y-x) )} \, \dd y \right| \\
& \leqs C \lambda^d \int_{ | x' - y' | \leqs \delta } \sum_{n = 0}^k \lambda^{n} \sup_{| x_j - y_j | \leqs \delta} |(\partial_j^n\psi)( \lambda (y-x) )| \, |\xi_j - \lambda^{-m} \partial_j \left( \fy( \lambda y) \right) |^{n - 2k} \\
& \qquad \qquad \qquad \qquad \qquad \qquad \qquad \qquad \qquad \qquad \qquad  
\times \lambda^{m (n-2k)} \, \dd y' \\
& \leqs C_k \ep^{- 2 k} \sum_{n = 0}^k \lambda^{d + n + m (n-2k)} \\
& \leqs C_{k,\ep} \lambda^{d - k(m-1)}. 
\end{aligned}
\end{equation*}
This shows that $(x_0,\xi_0) \notin \WF_{\rm g}^{m-1} (u)$ and the inclusion \eqref{eq:chirpconclusion1} follows. 

Next let $d = 1$. 
If $\fy$ is even then $u$ is even, and $W = -W$ since $m$ is even, so by \eqref{eq:evensteven1} we have either $\WF_{\rm g}^{m-1} (u) =\emptyset$ or $\WF_{\rm g}^{m-1} (u) = W$. 
The former is not true since $u \notin \cS(\ro)$. Thus we have proved \eqref{eq:chirpconclusion2} when $\fy$ is even. 

If $\fy$ is odd then $m$ is odd and $\check u (x) = \overline{ u(x) } = e^{- i  \fy(x)}$. 
Again $\WF_{\rm g}^{m-1} (u) =\emptyset$ cannot hold since $u \notin \cS(\ro)$. 
If we assume that the inclusion \eqref{eq:chirpconclusion1} is strict we  get a contradiction from 
\eqref{eq:evensteven0} and \eqref{eq:evensteven2}. 
Indeed suppose e.g. 
\begin{equation*}
\WF_{\rm g}^{m-1} (u) 
= \{ (x, \fy_m'(x) ) \in \rr 2: \ x > 0 \}. 
\end{equation*}
By \eqref{eq:evensteven0} and \eqref{eq:evensteven2} we then get the contradiction 
\begin{align*}
\WF_{\rm g}^{m-1} ( \check u) 
& = \{ (x, - \fy_m'(x) ) \in \rr 2: \ x < 0 \} \\
& = \{ (x, - \fy_m'(x) ) \in \rr 2: \ x > 0 \}
= \WF_{\rm g}^{m-1} ( \overline{u} ).   
\end{align*}
This proves \eqref{eq:chirpconclusion2} when $\fy$ is odd. 
\end{proof}

We would also like to determine $\WFgs (u)$ when $s \neq m-1$. 
The following two results treat this question. 

\begin{prop}\label{prop:chirpnegative1}
If $m \geqs 2$, $s > m-1$, and $\fy$ is a real polynomial defined by \eqref{eq:phasefunction}, \eqref{eq:polynomial1} and \eqref{eq:principalpart},
and $u$ is defined by \eqref{eq:chirpdef}, then 
\begin{equation}\label{eq:chirpconclusion3}
\WFgs (u) \subseteq (\rr d \setminus 0) \times \{ 0 \}. 
\end{equation}
If $d = 1$ and $\fy$ is even or odd then 
\begin{equation}\label{eq:chirpconclusion4}
\WFgs (u) 
= (\ro \setminus 0) \times \{ 0 \}.
\end{equation}
\end{prop}

\begin{proof}
Suppose $(x_0, \xi_0) \in T^* \rr d$ and $\xi_0 \neq 0$, that is $\xi_{0,j} \neq 0$ for some $1 \leqs j \leqs d$. 
From \eqref{eq:phasederivative1} we obtain
\begin{equation*}
\lambda^{-1-s} \partial_j \left( \fy( \lambda y) \right)
= \lambda^{m-1-s} \left(  \partial_j \fy_m (y) + \lambda^{1-m} \partial_j p(\lambda y) \right).
\end{equation*}
Thus from $s > m-1$, using \eqref{eq:polderivative1},
it follows that there exists $U \subseteq \rr {2d}$ such that $(x_0, \xi_0) \in U$, 
and $0 < \ep \leqs 1$, 
$L \geqs 1$ such that 
\begin{equation*}
| \xi_j - \lambda^{-1-s} \partial_j \left( \fy( \lambda y) \right)| \geqs \ep
\end{equation*}
when $(x,\xi) \in U$, $| x - y| \leqs \sqrt 2$ and $\lambda \geqs L$. 

Let $\psi \in C_c^\infty(\ro) \setminus 0$ be such that $\supp \psi \subseteq \rB_1$. 
The stationary phase theorem \cite[Theorem~7.7.1]{Hormander0}
yields, for any $k \in \no$, and any $\lambda \geqs L$, if $(x,\xi) \in U$, 
\begin{equation*}
\begin{aligned}
|V_\psi u ( \lambda x, \lambda^s \xi)|
& = (2 \pi)^{-\frac{d}{2}} \left| \int_{\rr d} e^{i ( \fy(y)  - \lambda^s \la y, \xi \ra )} \overline{\psi( \lambda (\lambda^{-1} y-x) )} \, \dd y \right| \\
& = (2 \pi)^{-\frac{d}{2}} \lambda^d \left| \int_{\rr d} e^{i \lambda^{1+s} ( \lambda^{-1-s} \fy(\lambda y) - \la y, \xi \ra )} \overline{\psi( \lambda (y-x) )} \, \dd y \right| \\
& \leqs C \lambda^d \int_{ | x' - y' | \leqs 1} \sum_{n = 0}^k \lambda^{n} \sup_{|x_j-y_j| \leqs 1} |(\partial_j^n\psi)( \lambda (y-x) )| \, |\xi_j - \lambda^{-1-s} \partial_j \left( \fy( \lambda y) \right) |^{n - 2k} \\
& \qquad \qquad \qquad \qquad \qquad \times \lambda^{(1+s) (n-2k)} \, \dd y' \\
& \leqs C_k \lambda^{d- k s}  \, \ep^{- 2 k}. 
\end{aligned}
\end{equation*}
This shows that $(x_0,\xi_0) \notin \WFgs(u)$ and \eqref{eq:chirpconclusion3} follows. 

When $d = 1$ and $\fy$ is either even or odd then \eqref{eq:chirpconclusion4} follows as in the proof of Theorem \ref{thm:chirpWFgs}. 
\end{proof}

\begin{prop}\label{prop:chirpnegative2}
Let $m \geqs 2$, $0 < s < m-1$, and $\fy$ be a real polynomial defined by \eqref{eq:phasefunction}, \eqref{eq:polynomial1} and \eqref{eq:principalpart}. 
Suppose $\fy_m (x) \neq 0$ for all $x \in \rr d \setminus 0$.
If $u$ is defined by \eqref{eq:chirpdef} then 
\begin{equation}\label{eq:chirpconclusion5}
\WFgs (u) \subseteq  \{ 0 \} \times (\rr d \setminus 0). 
\end{equation}
If $d = 1$ and $\fy$ is even then 
\begin{equation}\label{eq:chirpconclusion6}
\WFgs (u) 
= \{ 0 \} \times (\ro \setminus 0).
\end{equation}
\end{prop}

\begin{proof}
Suppose $(x_0, \xi_0) \in T^* \rr d$ and $x_0 \neq 0$. 
The assumption $\fy_m (x) \neq 0$ for all $x \in \rr d \setminus 0$ and Euler's homogeneous function theorem 
imply that $\nabla \fy_m (x_0) \neq 0$, that is $\partial_j \fy_m (x_0) \neq 0$ for some $1 \leqs j \leqs d$. 
From \eqref{eq:phasederivative1} and \eqref{eq:polderivative1}
and $s < m-1$
it follows that there exists $U \subseteq \rr {2d}$ such that $(x_0, \xi_0) \in U$, 
$1 \leqs j \leqs d$
and $0 < \ep \leqs 1$, 
$L \geqs 1$ such that 
\begin{equation*}
| \lambda^{1 + s - m} \xi_j - \lambda^{-m} \partial_j \left( \fy( \lambda y) \right)| \geqs \ep
\end{equation*}
when $(x,\xi) \in U$, $| x - y| \leqs \ep \sqrt 2$ and $\lambda \geqs L$. 

Let $\psi \in C_c^\infty(\ro) \setminus 0$ be such that $\supp \psi \subseteq \rB_\ep$. 
Again by the stationary phase theorem \cite[Theorem~7.7.1]{Hormander0}
we obtain, for any $k \in \no$, and any $\lambda \geqs L$, if $(x,\xi) \in U$, 
\begin{equation*}
\begin{aligned}
|V_\psi u ( \lambda x, \lambda^s \xi)|
& = (2 \pi)^{-\frac{d}{2}} \left| \int_{\rr d} e^{i ( \fy(y)  - \lambda^s \la y, \xi \ra )} \overline{\psi( \lambda (\lambda^{-1} y-x) )} \, \dd y \right| \\
& = (2 \pi)^{-\frac{d}{2}} \lambda^d \left| \int_{\rr d} e^{i \lambda^m ( \lambda^{-m} \fy(\lambda y) - \lambda^{1+s-m} \la y, \xi \ra )} \overline{\psi( \lambda (y-x) )} \, \dd y \right| \\
& \leqs C \lambda^d \int_{ | x' - y' | \leqs \ep}  \sum_{n = 0}^k \lambda^{n} \sup_{|x_j-y_j| \leqs \ep} |(\partial_j^n\psi)( \lambda (y-x) )| \, | \lambda^{1 + s - m} \xi_j - \lambda^{-m} \partial_j \left( \fy( \lambda y) \right) |^{n - 2k} \\
& \qquad \qquad \qquad \qquad \qquad \times \lambda^{m (n-2k)} \, \dd y' \\
& \leqs C_k \lambda^{d- k(m-1) }  \, \ep^{- 2 k}. 
\end{aligned}
\end{equation*}
This shows that $(x_0,\xi_0) \notin \WFgs(u)$ and \eqref{eq:chirpconclusion3} follows. 

When $d = 1$ and $\fy$ is even then \eqref{eq:chirpconclusion6} follows as in the proof of Theorem \ref{thm:chirpWFgs}. 
\end{proof}

\begin{example}\label{ex:airytype}
Let $k \in \no \setminus 0$ and consider the differential equation
\begin{equation*}
u^{(k)} - x u = 0. 
\end{equation*}
for $u \in \cS'(\ro)$. 
When $k = 2$ this is the Airy equation. 
Fourier transformation gives
\begin{equation}\label{eq:airytypefourier}
i^k \xi^k \wh u + D \wh u = 0 
\end{equation}
which is solved by 
\begin{equation*}
\wh u (\xi) = C \exp \left( -i^{k+1} \frac{\xi^{k+1}}{k+1} \right), \quad C \in \co. 
\end{equation*}
This function belongs to $\cS'(\ro)$ provided $k \notin 1 + 4 \no$. 

The equation \eqref{eq:airytypefourier} can be written $a^w(x,D) \wh u = 0$
where 
\begin{equation*}
a(x,\xi) = i^k x^k + \xi. 
\end{equation*}

By Example \ref{ex:polynomial} we know that $a \in G^k \cap G^{k,k}$. 
Suppose $k = 2 n$ with $n \in \no \setminus 0$. 
Since $a(x,\xi) = 0$ when $\xi = (-1)^{n+1} x^{2n}$, 
it follows from Definition \ref{def:noncharacteristic} that $(x,(-1)^{n+1} x^{2n}) \in \charac_{2n} (a)$
for any $x \neq 0$. 
It holds 
\begin{equation}\label{eq:charequality1}
\charac_{2n} (a) = \{ (x,(-1)^{n+1} x^{2n}) \in \rr 2, \ x \neq 0\}.
\end{equation}
In fact it suffices to show 
\begin{equation}\label{eq:charinclusion1}
\charac_{2n} (a) \subseteq \{ (x,(-1)^{n+1} x^{2n}) \in \rr 2, \ x \neq 0\}. 
\end{equation}

Suppose $(x_0,\xi_0) \in \sr 1$ with $\xi_0 \neq (-1)^{n+1} x_0^{2n}$. 
In order to show \eqref{eq:charinclusion1} we must show $(x_0,\xi_0) \notin \charac_{2n} (a)$. 
There exist $\ep, \delta > 0$ such that
\begin{equation*}
| \xi - (-1)^{n+1} x^{2n} | 
\geqs \delta ( |x| + |\xi|^{\frac{1}{2n}} )^{2n}
\end{equation*}
if $(x,\xi) \in (x_0, \xi_0) + \rB_\ep$. 
This inequality is $2n$-conic, that is invariant to the transformation $T^* \ro \setminus 0 \ni ( x, \xi) \mapsto ( \lambda x, \lambda^{2n} \xi)$ for $\lambda > 0$. 
It follows that 
\begin{equation*}
| a(x,\xi) | 
\geqs \delta ( |x| + |\xi|^{\frac{1}{2n}} )^{2n}
\end{equation*}
when $(x,\xi) \in \wt \Gamma_{2n, (x_0,\xi_0),\ep}$. 
Hence from Lemma \ref{lem:sconesequiv} it follows $(x_0,\xi_0) \notin \charac_{2n} (a)$ and we have shown \eqref{eq:charinclusion1}
and thereby \eqref{eq:charequality1}. 

Invoking Theorem \ref{thm:microelliptic} we obtain
\begin{align*}
\WF_{\rm g}^{2n} ( \wh u) 
& \subseteq \WF_{\rm g}^{2n} (a^w(x,D) \wh u)  \bigcup \charac_{2n} (a) \\
& = \charac_{2n} (a) =  \{ (x,(-1)^{n+1} x^{2n}) \in \rr 2, \ x \neq 0\}. 
\end{align*}
Thus we have found an alternative proof of a particular case of the inclusion \eqref{eq:chirpconclusion1} in Theorem \ref{thm:chirpWFgs} when $m$ is odd. 
From \eqref{eq:chirpconclusion2} we know that the inclusion is actually an equality.  

Adding this information and applying Proposition \ref{prop:sGabormetaplectic} (i) we obtain
\begin{equation*}
\WF_{\rm g}^{\frac{1}{2n}} ( u) 
= - \J \WF_{\rm g}^{2n} ( \wh u)
= \{ ((-1)^{n} x^{2n},x) \in \rr 2, \ x \neq 0\}. 
\end{equation*}

If $n=1$ then $u$ is the Airy function (multiplied by $C$) \cite{Hormander0}, 
and thus
\begin{equation*}
\WF_{\rm g}^{\frac{1}{2}} ( u) 
= \{ ( -x^{2},x) \in \rr 2, \ x \neq 0\}. 
\end{equation*}

This can be compared to \cite[Example~8.5]{PRW1} 
which says that 
\begin{equation*}
\WFg ( u) 
= \WF_{\rm g}^{1} ( u) 
= \{ (x,0) \in \rr 2, \ x < 0\}. 
\end{equation*}
\end{example}

\section*{Acknowledgment}
Work partially supported by the MIUR project ``Dipartimenti di Eccellenza 2018-2022'' (CUP E11G18000350001).


\end{document}